\documentclass[11pt,a4paper]{article}

\usepackage[T1]{fontenc}
\usepackage{lmodern}
\usepackage[utf8]{inputenc}

\usepackage[left=2.45cm, top=2.45cm,bottom=2.45cm,right=2.45cm]{geometry}
\usepackage{amsmath}
\usepackage{amssymb,amsthm,graphicx,xcolor,tikz,pgfplots,bm,bbm}

\usepackage[british]{babel}
\usepackage{enumerate}

\usepackage[
	bookmarks=true,
	bookmarksnumbered=true,
	bookmarksopen=true,
	unicode=true,
	pdftoolbar=true,
	pdfmenubar=true,
	pdffitwindow=false,
	pdfstartview={FitH},
	pdftitle={},
	pdfauthor={},
	pdfsubject={},
	pdfcreator={},
	pdfproducer={},
	pdfkeywords={},
	pdfnewwindow=true,
	colorlinks=true,
	linkcolor=black,
	citecolor=black,
	filecolor=black,
	urlcolor=black]{hyperref}
	
\usepackage{todonotes}

\makeatletter
	\def\@cite#1#2{[\textbf{#1}\if@tempswa , #2\fi]}	
	\def\@biblabel#1{[#1]}								
\makeatother

\newtheorem {theorem}{Theorem}[section]
\newtheorem {proposition}[theorem]{Proposition}
\newtheorem {lemma}[theorem]{Lemma}

\theoremstyle{definition}
\newtheorem{definition}[theorem]{Definition}
\newtheorem {remark}[theorem]{Remark}
\newtheorem {example}[theorem]{Example}

\newcommand{\EE}{\mathbb{E}}
\newcommand{\HH}{\mathbb{H}}
\newcommand{\NN}{\mathbb{N}}
\newcommand{\PP}{\mathbb{P}}
\newcommand{\RR}{\mathbb{R}}
\renewcommand{\SS}{\mathbb{S}}

\newcommand{\cH}{\mathcal{H}}

\DeclareMathOperator{\Var}{Var}

\DeclareMathOperator{\Wass}{Wass}
\DeclareMathOperator{\Kol}{Kol}

\DeclareMathOperator{\Hyp}{Hyp}

\DeclareMathOperator{\cum}{cum}
\DeclareMathOperator{\arcosh}{arcosh}
\newcommand{\bo}{\mathbf{o}}
\newcommand{\B}{B}

\newcommand{\eps}{\varepsilon}
\newcommand{\dint}{\mathrm{d}}
\newcommand{\id}{\mathbbm 1}

\usepackage{accents}
\DeclareMathSymbol{\widetildesym}{\mathord}{largesymbols}{"65}

\begin{document}

\title{\bfseries Fluctuations of $\lambda$-geodesic\\ Poisson hyperplanes in hyperbolic space}

\author{%
    Zakhar Kabluchko\footnotemark[1]%
    \and Daniel Rosen\footnotemark[2]%
    \and Christoph Th\"ale\footnotemark[3]%
}

\date{}
\renewcommand{\thefootnote}{\fnsymbol{footnote}}
\footnotetext[1]{%
    University of M\"unster, Germany. Email: zakhar.kabluchko@uni-muenster.de
}

\footnotetext[2]{%
	Ruhr University Bochum, Germany. Email: daniel.rosen@rub.de
}	

\footnotetext[3]{%
    Ruhr University Bochum, Germany. Email: christoph.thaele@rub.de
}

\maketitle

\begin{abstract}\noindent
Poisson processes of so-called $\lambda$-geodesic hyperplanes in $d$-dimensional hyperbolic space are studied for $0\leq\lambda\leq 1$. The case $\lambda=0$ corresponds to genuine geodesic hyperplanes, the case $\lambda=1$ to horospheres and $\lambda\in(0,1)$ to $\lambda$-equidistants. In the focus are the fluctuations of the centred and normalized total surface area of the union of all $\lambda$-geodesic hyperplanes in the Poisson process within a hyperbolic ball of radius $R$ centred at some fixed point, as $R\to\infty$. It is shown that for $\lambda<1$ these random variables satisfy a quantitative central limit theorem precisely for $d=2$ and $d=3$. The exact form of the non-Gaussian, infinitely divisible limiting distribution is determined for all higher space dimensions $d\geq 4$. The special case $\lambda=1$ is in sharp contrast to this behaviour. In fact, for the total surface area of Poisson processes of horospheres, a non-standard central limit theorem with limiting variance $1/2$ is established for all space dimensions $d\geq 2$. We discuss the analogy between the problem studied here and the Random Energy Model whose partition function exhibits a similar structure of possible limit laws.

    \smallskip\noindent
    \textbf{Keywords.} Central limit theorem, horospheres, hyperbolic stochastic geometry, $\lambda$-geodesic hyperplanes, non-central limit theorem, Poisson hyperplane process, Poisson point process, Random Energy Model.

    \smallskip\noindent
    \textbf{MSC 2010.} Primary  52A55, 60D05; Secondary 60F05, 60G55.
\end{abstract}


\section{Introduction and motivation}

The probabilistic understanding of random geometric systems is the main goal of stochastic geometry. In the past, models for such systems have typically been investigated in Euclidean spaces. However, recently there has been a growing interest in random geometric systems in non-Euclidean spaces, most prominently, the spaces of constant negative curvature $-1$. Examples include the study of random hyperbolic Voronoi tessellations \cite{BenjaminiPaquettePfeffer,HansenMueller,GodlandKabluchkoThaele},  hyperbolic random polytopes \cite{BesauRosenThaele,BesauThaele}, the hyperbolic Boolean model \cite{BenjaminiJonassonSchrammTykesson,Tykesson,TykessonCalka}, hyperbolic Poisson cylinder processes \cite{BromanTykesson} or hyperbolic random geometric graphs \cite{FountoulakisYukich,OwadaYogesh}. Most closely related to the present paper are the investigations in \cite{HeroldHugThaele} about Poisson processes of hyperplanes, that is, totally geodesic $(d-1)$-dimensional submanifolds, in $d$-dimensional hyperbolic space. While first-order properties,{  i.e. expectations}, of geometric functionals associated with such processes are independent of the curvature of the underlying space, it turned out that second-order parameters,{  e.g., variances and correlations,} and also the accompanying central limit theory are rather sensitive to curvature. For example, it has been shown that the total surface content or the number of intersection points of such Poisson hyperplane processes that can be observed in a sequence of growing hyperbolic balls satisfy a central limit theorem only in space dimensions $d=2$ and $d=3$. This is in sharp contrast to the situation in Euclidean space, where a central limit theorem for these geometric quantities holds in all space dimensions, see \cite{Heinrich,LPST,RS}.

\begin{figure}[t]
	\begin{center}
		\includegraphics[width=0.45\columnwidth]{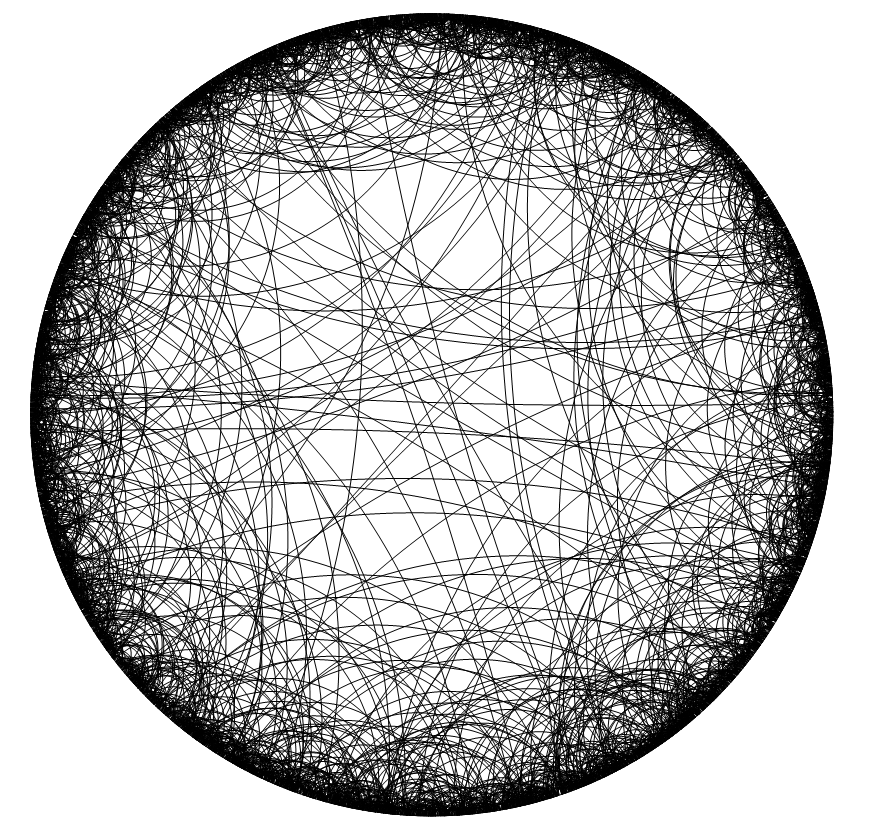}\quad
		\includegraphics[width=0.43\columnwidth]{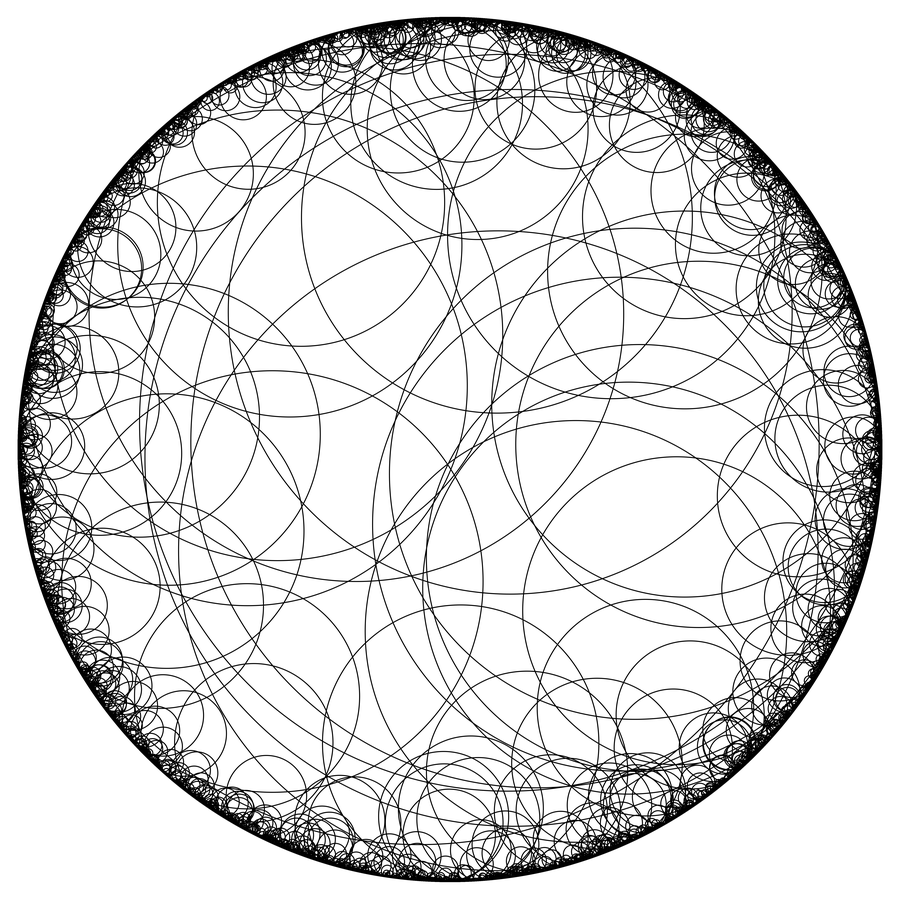}
		\label{fig:Simulations}
		\caption{Left panel: Simulation of a hyperbolic Poisson hyperplane process in the hyperbolic plane (corresponding to the choice $\lambda=0$). Right panel: Simulation of a Poisson process of horospheres in the hyperbolic plane (corresponding to the choice $\lambda=1$).}
	\end{center}
\end{figure}

The purpose of the present paper is twofold. Our first main point is that the notion of a hyperplane in Euclidean space does not have a unique or canonical analogue in hyperbolic space.  Rather, there is
a whole family of such analogues which is parametrized by a real number $0\leq\lambda\leq 1$. Let us briefly explain this (we refer to Section \ref{subsec:lambdaHyperplanes} below for more details and formal definitions). We first recall that a hyperplane in a Euclidean space is not only characterized by being totally geodesic, but also as unbounded totally \textit{umbilic} $(d-1)$-dimensional submanifolds ({i.e. submanifolds whose second fundamental form is everywhere a scalar multiple of the Riemannian metric, see Section \ref{subsec:lambdaHyperplanes} for more details}). However, in hyperbolic geometry, apart from geodesic hyperplanes there are many more unbounded totally umbilic hypersurfaces, {distinguished by their (necessarily constant) normal curvature $\lambda \geq 0 $}: horospheres (or horocycles in dimension $2$) for $\lambda=1$, equidistants for $0<\lambda<1$ and totally geodesic hyperplanes for $\lambda=0$. Collectively, they have been introduced in \cite{Sol} as $\lambda$-geodesic hyperplanes, for $\lambda\in[0,1]$.
To be concrete, in the upper half-plane model for the hyperbolic plane $\HH^2$, horocycles are Euclidean circles tangent to the boundary line or Euclidean lines parallel to it, while  equidistants are realized as intersections with the upper half-plane of Euclidean circles or lines intersecting the boundary line at an angle $\theta$, where $\cos \theta=\lambda\in(0,1)$. Moreover, geodesic lines correspond to infinite rays orthogonal to the boundary line or to Euclidean half-circles intersecting the boundary line orthogonally, see Figure \ref{fig:LambdaGeodesics}. In view of this,  it is natural to consider not only Poisson processes of totally geodesic submanifolds of $\HH^d$ as in \cite{HeroldHugThaele} or \cite{SantaloYanez} for $d=2$, but more generally Poisson processes of $\lambda$-geodesic hyperplanes for all values $\lambda\in[0,1]$.

For such process of $\lambda$-geodesic hyperplanes we shall study the total $(d-1)$-volume in a sequence of growing balls $B_R^d$ centred at some arbitrary but fixed point in $\HH^d$. That is, we consider the sum over all $\lambda$-geodesic hyperplanes of the process of the $(d-1)$-volume of the intersection with $B_R^d$. In particular, we are interested in the fluctuations of this family of random variables as $R\to\infty$, after suitable centring and renormalization. To describe our results, we need to distinguish the cases $\lambda\in[0,1)$ and $\lambda=1$. The reason behind this lies in the fact that the intrinsic geometry of a $\lambda$-geodesic hyperplane is again hyperbolic (with constant sectional curvature equal to $-(1-\lambda^2)$) as long as $\lambda\in[0,1)$, while the intrinsic geometry of horospheres corresponding to the case $\lambda=1$ is Euclidean (with constant sectional curvature equal to $0$). As in \cite{HeroldHugThaele}, we shall prove for $\lambda\in[0,1)$ that, as $R\to\infty$, our target random variables satisfy a central limit theorem only for $d=2$ and $d=3$. More specifically, we develop a quantitative central limit theorem of Berry-Esseen-type.  The second goal of this paper is to characterize for general $\lambda\in[0,1)$ the non-Gaussian limit distribution for all higher space dimensions $d\geq 4$. As it turns out, the limit distribution is infinitely divisible with vanishing Gaussian component.  We will be able to determine explicitly its characteristic function and, in particular, all of its cumulants in terms of the parameter $\lambda$. However, in the special case $\lambda=1$ this approach breaks down and in fact we shall argue that in this case we have a central limit theorem for all space dimensions $d\geq 2$, which in a sense resembles the situation for Poisson hyperplane processes in Euclidean spaces. However, in the hyperbolic set-up, the central limit theorem has a distinguished feature, namely that the limiting variance is not equal to $1$, as could be expected, but is equal to $1/2$. This structure of possible limiting laws (normal, infinitely divisible without Gaussian component, and normal with non-standard variance $1/2$ in the boundary case) is strongly reminiscent of what is known on the fluctuations of the partition function of the Random Energy Model; see~\cite[Theorems~1.5, 1.6]{BovierKurkovaLoewe} and~\cite[Chapter~9]{BovierBook} (in particular, Theorem~9.2.1 there).  This similarity will be discussed in more detail at the end of this paper in Section~\ref{sec:REM}.

{The proofs of the central limit theorems rely on quantitative normal approximation bounds arising from the Malliavin-Stein method \cite{RS,SchulteKolmogorov}, similarly to \cite{HeroldHugThaele}; for the non-central limit theorems we use characteristic function techniques. In both cases, the probabilistic tools are combined with geometric estimates for intersection volumes, which are new in the case $\lambda > 0$.}

\medspace

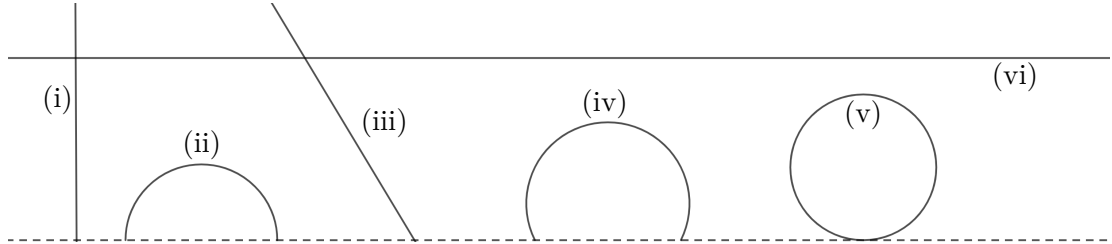
\begin{figure}[t]
	\begin{center}
		\begin{tikzpicture}
			\node[inner sep=0pt] (russell) at (0,0)
			{\includegraphics[width=0.9\columnwidth]{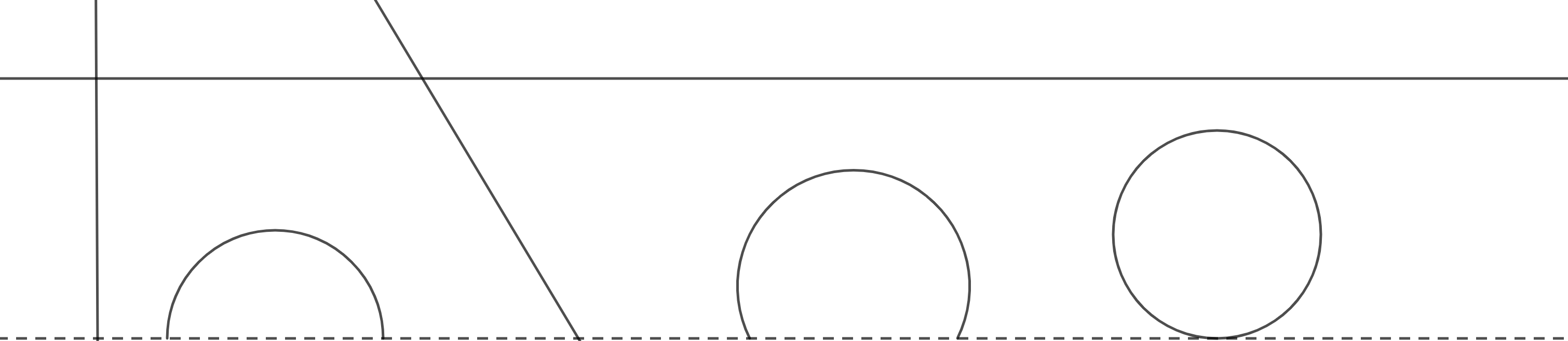}};
			\node at (-6.6,0.3) {(i)};
			\node at (-4.7,-0.3) {(ii)};
			\node at (-2.3,0) {(iii)};
			\node at (0.6,0.25) {(iv)};
			\node at (4,0.1) {(v)};
			\node at (6,0.6) {(vi)};
		\end{tikzpicture}
		\label{fig:LambdaGeodesics}
		\caption{$\lambda$-geodesic lines in the upper half-plane model for a hyperbolic plane: (i) and (ii) correspond to $\lambda=0$ (geodesic lines), (iii) and (iv) to $0<\lambda<1$ (equidistants), (v) and (vi) to $\lambda=1$ (horocycles).}
	\end{center}
\end{figure}

The remaining parts of this paper are organized as follows. In Section \ref{sec:Results} we formally introduce the models we consider and state our results. More precisely, in Section \ref{subsec:Hyperplanes} we determine the limiting distribution in the non-central limit theorem for the total surface area of genuine geodesic hyperplanes and in Section \ref{subsec:lambdaHyperplanes} we introduce general $\lambda$-geodesic hyperplanes and present our results for them. The proof for geodesic hyperplanes is the content of Section \ref{sec:ProofHyperplanes}. In Section \ref{sec:background} we recall some background material about the geometry of $\lambda$-geodesic hyperplanes and derive geometric estimates for them which are needed for our arguments. The proof of our main result on the fluctuations of the total surface area for $\lambda$-geodesic hyperplanes is presented in Section \ref{subsec:ProofLambdad<4} (considering the case $d \leq 3$)
and Section \ref{sec:ProofLambdad>4} (concerning the case $d\geq 4$). The non-standard central limit theorem for horospheres is the content of Section \ref{sec:ProofHoro}. Finally, the analogy with the Random Energy Model is the subject of Section~\ref{sec:REM} which also contains a non-rigorous discussion of all previous results.

\paragraph{Notation and terminology.} Throughout this paper, $\HH^d$ denotes the $d$-dimensional hyperbolic space, that is, the unique complete and simply connected Riemannian manifold of constant sectional curvature  $-1$. Concrete models will be recalled below, as needed. We write $\xrightarrow{D}$ for convergence of random variables in distribution, {and denote by $\cum_\ell(\,\cdot\,)$ the $\ell$-th cumulant of a random variable}. Given two real-valued functions $a(R)$ and $b(R)$, we use the notation $a \sim b$ as $R \to \infty$ to indicate that $\lim_{R\to\infty} \frac{a(R)}{b(R)}  = 1$, while $a \asymp b$ as $R \to \infty$ means that the ratio $\frac{a(R)}{b(R)}$ is bounded between two positive constants for sufficiently large $R$.  {Finally, we shall write $\kappa_d$ and $\omega_d$ for the volume and the surface area of a $d$-dimensional Euclidean unit ball, respectively.}

\section{Models and results}\label{sec:Results}

\subsection{The case of geodesic hyperplanes}\label{subsec:Hyperplanes}

We consider a $d$-dimensional hyperbolic space $\HH^d$ with $d\geq 2$,  and let $\Hyp$ be the space of hyperbolic hyperplanes in $\HH^d$, that is, the space of all $(d-1)$-dimensional totally geodesic {submanifolds} of $\HH^d$. Let $\Lambda$ denote the (infinite) measure on $\Hyp$ which is invariant under all isometries of $\HH^d$. We remark that $\Lambda$ is unique up to a multiplicative constant which will be specified  in what follows.

To describe the measure $\Lambda$,  fix an origin $\bo\in \HH^d$. We parametrize those $H\in\Hyp$ which do not pass through $\bo$ by a pair $(s, u) \in (0,\infty) \times\SS^{d-1}$, where $s\in (0,\infty)$ is the (hyperbolic) distance from $H$ to $\bo$, and $u\in\SS^{d-1}$ is the unit vector ({with respect to the hyperbolic Riemannian metric on the tangent space $T_\bo\HH^d$}) {pointing towards $H$}. The invariant measure $\Lambda$ on $\Hyp$ is then defined by the relation
\begin{equation}\label{eq:dH}
	\int_{\Hyp}f(H)\,\Lambda(\dint H) = {2}\int_{0}^\infty\int_{\SS^{d-1}} f(H(s,u))\,\cosh^{d-1}(s) \,\dint u\,\dint s,
\end{equation}
where $f:\Hyp\to\RR$ is a non-negative measurable function and $H(s,u)$ stands for the unique element of $\Hyp$ parametrized by $(s,u)$ as described above; see \cite[Equation (17.41)]{Santalo} and also the discussion around \cite[Equation (6)]{HeroldHugThaele}. Here, $\dint s$ and $\dint u$ refer to integration with respect to the {Lebesgue measure on $\RR$ and the normalized surface measure on $\SS^{d-1}$, respectively.}

\begin{example}
{The total invariant measure of hyperplanes intersecting a hyperbolic ball of radius $R$ is obtained by setting $f = \id \{s \leq R\}$ in \eqref{eq:dH}:
\begin{equation*}
\Lambda\big(\{H \in \Hyp \,:\, H \cap B_R \neq \emptyset  \}\big) = 2\int_0^R \cosh^{d-1}(s) \,\dint s.
\end{equation*}
In general, the total invariant measures of hyperplanes meeting a convex body with smooth boundary in $\HH^d$ is given as a certain linear combination of its volume and mean curvature integrals of its boundary (see \cite[p. 310]{Santalo}).}
\end{example}

Now, let $\eta$ be a Poisson process on $\Hyp$ with intensity measure $\Lambda$. We are interested in the total surface area
$$
S_R:=\sum_{H\in\eta}\cH^{d-1}(H\cap B_R^d)
$$
of $\eta$ within a hyperbolic ball $B_R^d$ of radius $R>0$ around an arbitrary but fixed point $p\in\HH^d$. Here, $\cH^{d-1}$ stands for the $(d-1)$-dimensional hyperbolic Hausdorff measure. It has been shown in \cite[Theorem 5]{HeroldHugThaele} that the normalized surface area $(S_R-\EE S_R)/\sqrt{\Var S_R}$ satisfies, as $R\to\infty$, a central limit theorem only if $d=2$ or $d=3$, whereas for $d\geq 4$ a potential limit distribution was shown to be necessarily non-Gaussian. However, it remained open in \cite{HeroldHugThaele} to determine this limiting distribution for $d\geq 4$. The first purpose of this paper is to fill this gap. 

\begin{theorem}\label{thm:NonClt}
Suppose that $d\geq 4$. Then,
$$
{S_R-\EE S_R\over e^{R(d-2)}} \overset{D}{\longrightarrow} {\omega_{d-1}\over (d-2)2^{d-2}}\,  Z_d,\qquad\text{as}\;\, R\to\infty,
$$
where $Z_d$ is an infinitely divisible, zero-mean random variable defined by
$$
Z_d:= \lim_{T\to+\infty} \Bigg(\sum_{\substack{s\in\zeta_d\\ s\leq T}} \cosh^{-(d-2)}(s) - \textcolor{black}{2}\sinh T\Bigg)
\quad
\text{ in $L^2$ and a.s.,}
$$
and $\zeta_d$ is an inhomogeneous Poisson process on $[0,\infty)$ with density function $s\mapsto \textcolor{black}{2}\cosh^{d-1}(s)$.
\end{theorem}

\begin{remark}
\begin{itemize}
\item[(1)]
By \cite[Lemma 11]{HeroldHugThaele} the rescaling $e^{R(d-2)}$ in the previous theorem is of the same order as $\sqrt{\Var S_R}$, as $R\to\infty$, up to a multiplicative constant only depending on $d$.
\item[(2)]
Alternatively, one might consider for $t>0$ a Poisson process $\eta_t$ on the space $\Hyp$ with intensity measure $t\Lambda$, in which case the parameter $t$ has an interpretation as an intensity. Then, for \textit{fixed} radius $R>0$ one might ask for the fluctuations of the total surface area, as $t\to\infty$. In this case a central limit theorem for all space dimensions $d\geq 2$ has been shown in \cite[Theorem 4]{HeroldHugThaele}. {In fact, it holds for all intersection processes and} follows from general results for so-called Poisson U-statistics \cite{LPST,RS}. A similar comment also applies to the case of Poisson processes of $\lambda$-geodesic hyperplanes considered in the next subsection.
\end{itemize}
\end{remark}

\begin{remark}\label{rem:Zd}
	Let us discuss existence and  properties of the random variable $Z_d$ appearing in Theorem~\ref{thm:NonClt}.  For $T>0$, the expectation of the random variable $Y_T := \sum_{\substack{s\in\zeta_d, s\leq T}} \cosh^{-(d-2)}(s)$ is {computed with the help of the Mecke equation (see, e.g., \cite[Theorem 4.1]{LPBook})} and equals $\EE Y_T = \textcolor{black}{2}\int_0^T \cosh s \, \dint s = \textcolor{black}{2}\sinh T$, while its variance is {similarly computed as} 
$$
\Var Y_T = \textcolor{black}{2}\int_{0}^T \cosh^{(d-1) - 2(d-2)}(s) \,\dint s = \textcolor{black}{2}\int_{0}^T \cosh^{3-d}(s) \,\dint s.
$$
Note that the latter integral converges (and stays bounded) as $T\to+\infty$ by the assumption $d\geq 4$.  Therefore, $Z_d := \lim_{T\to +\infty} (Y_T - \EE Y_T)$ exists a.s.\ and in $L^2$ by the martingale convergence theorem for $L^2$-bounded martingales. A slight generalization of the above argument (using, e.g., \cite[Corollary 1]{LPST}) shows that all cumulants (and hence, all moments) of $Y_T - \EE Y_T$ stay bounded as $T\to\infty$ and hence, $Y_T - \EE Y_T \to Z_d$ in $L^p$ for all $p>0$.
The cumulants of the limiting random variable $Z_d$ are given as follows. We have $\cum_1(Z_d)= \EE Z_d = 0$ and
$$
\cum_\ell(Z_d) = \lim_{T\to\infty} \cum_\ell(Y_T- \EE Y_T) =  \textcolor{black}{2}\int_0^\infty \cosh^{-( (d-2)\ell -(d-1))} (s)\,\dint s
=
\textcolor{black}{\sqrt{\pi}}{\Gamma\Big({(d-2)\ell-(d-1)\over 2}\Big)\over\Gamma\Big({(d-2)(\ell-1)\over 2}\Big)},
$$
for all $\ell \geq 2$,
where we used \cite[Corollary 1]{LPST} and then the formula $\int_{-\infty}^\infty (\cosh y)^{-h}\,\dint y=\sqrt{\pi}\,\Gamma(h/2)/ \Gamma((h+1)/2)$ for all $h>0$; see, e.g., \cite[Equation (3.14)]{KabluchkoAngles}.
In particular, $Z_d$ cannot be a Gaussian random variable, which would have vanishing cumulants for $\ell\geq 3$. For example, if $d=4$ we have $\cum_2(Z_4)=\textcolor{black}{\pi} $, $\cum_3(Z_4)=\frac \pi {\textcolor{black}{2}}$ and $\cum_4(Z_4)=\frac {3\pi}{\textcolor{black}{8}}$. Moreover, in the proof of Theorem \ref{thm:NonClt} we show that the random variable $Z_d$ has characteristic function given by
$$
\EE e^{\mathfrak{i} t Z_d}
=
\exp\Big(\textcolor{black}{2}\int_0^\infty(e^{\mathfrak{i}th(s)}-1-\mathfrak{i}th(s))\cosh^{d-1}(s)\,\dint s\Big),\qquad h(s) := \cosh^{-(d-2)}(s),
$$
where $\mathfrak{i}= \sqrt{-1}$ stands for the imaginary unit. Thus, $Z_d$ is infinitely divisible; see \cite[Corollary 7.6]{KallenbergVol1} or~\cite[Theorem~8.1]{Sato_book} for the  L\'evy-Khinchin formula and~\cite[Chapter~4]{Sato_book} for the L\'evy-It\^{o} decomposition. The L\'evy measure $\nu_d$ lives on $(0,1)$ and is the image of the measure $\textcolor{black}{2}\cosh^{d-1}(s)\dint s$ on $(0,\infty)$ under the map $s\mapsto h(s)$. Thus, the Lebesgue density of $\nu_d$ is given by
$$
\frac{\nu_d(\dint y)}{\dint y} = \frac{\textcolor{black}{2} \cosh^{d-1}(h^{-1}(s))}{|h'(h^{-1}(y))|}
=
\frac{\textcolor{black}{2}\cosh^{2d-2} (\arcosh (y^{-\frac 1{d-2}}))}{(d-2)\sinh(\arcosh (y^{-\frac 1{d-2}}))}
=
\frac{\textcolor{black}{2}}{(d-2) y^{\frac{2d-3}{d-2}} \sqrt{1 - y^{\frac 2 {d-2}}}}
,
$$
for  $y\in (0,1)$.
In particular, $Z_d$ has no Gaussian component, which once again shows that $Z_d$ is non-Gaussian. Since $\nu_d$ is supported on $(0,1)$, it follows that $Z_d$ has finite exponential moments, that is $\EE e^{c Z_d} <\infty$ for all $c\in \RR$; see~\cite[Theorem~26.1]{Sato_book}.
As $y\to 0$, the density of $\nu_d$ is asymptotically equivalent to $\frac {\textcolor{black}{2}}{d-2} y^{-\frac{2d-3}{d-2}}$and it follows from~\cite{Orey}, see also~\cite[Proposition~28.3]{Sato_book}, that $Z_d$ has an infinitely differentiable density and all of its derivatives vanish at $\infty$.
\end{remark}

\subsection{The case of \texorpdfstring{$\lambda$}{lambda}-geodesic hyperplanes}\label{subsec:lambdaHyperplanes}

In this section we propose a generalization of some of the results from \cite{HeroldHugThaele} and Theorem \ref{thm:NonClt} to so-called $\lambda$-geodesic hyperplanes. To introduce the model, we need to briefly recall some notions from the differential geometry of hypersurfaces. For more details we refer the reader to, e.g., \cite[Chapter 6]{doC} or  \cite[Chapter 8]{Lee}. Suppose that $H$ is a hypersurface of a Riemannian manifold $(M,g)$, equipped with a choice of a unit normal vector $n_H$. The second fundamental form of $H$ at a point $p \in H$ is the symmetric bilinear form $\B$ on the tangent space $T_pH$ defined by
\begin{equation*}
\B(X,Y) = g ( \widetilde{\nabla}_XY, n_H(p) ), \qquad X,Y \in T_pH,
\end{equation*}
where $\widetilde{\nabla}$ denotes the Levi-Civita connection of $M$. In other words, $B$ measures the normal component of the ambient covariant derivative. The second fundamental form relates curvature measurements performed in $H$ to those performed in $M$, in the sense that for any arclength parametrized curve $\gamma$ lying in $H$ one has 
\begin{equation*}
\widetilde{c}_\gamma^2 = c_\gamma^2 + B(\dot{\gamma}, \dot{\gamma})^2
\end{equation*}
(see, e.g., \cite[Proposition 8.10]{Lee}), where $\widetilde{c}_\gamma$ and $c_\gamma$ are the geodesic curvatures of $\gamma$ measured inside $M$ and $H$, respectively.

A point $p \in H$ is called \emph{umbilic} if the second fundamental form of $H$ at $p$ is a scalar multiple of the Riemannian metric, i.e.
\begin{equation}
\B(\,\cdot\,, \,\cdot\,) = \lambda \, g (\,\cdot\,, \,\cdot\,),
\end{equation}
and the factor $\lambda$ is called the \emph{normal curvature} at $p$ (note that the sign of $\lambda$ depends on the choice of $n_H$). The hypersurface $H$ is called \emph{totally umbilic} if all of its points are umbilic. It is known {(see \cite[Lemma 7.25]{Spivak})} that when the ambient manifold has constant curvature, the normal curvature of all points in a totally umbilic hypersurface is the same. For example, in Euclidean space the totally umbilic hypersurfaces are hyperplanes (with $\lambda = 0$) and spheres of radius $R  \in (0,\infty)$ (with $\lambda = \frac{1}{R}$). In particular, Euclidean hyperplanes are the only non-compact totally umbilic hypersurfaces of Euclidean space.

In hyperbolic space, totally umbilic hypersurfaces are again classified by their (constant) normal curvature $\lambda$ ({see \cite[Theorem 7.29]{Spivak}}): when $\lambda > 1$, these are hyperbolic spheres of radius ${\rm artanh} (1/\lambda)$; when $\lambda = 0$, these are totally geodesic hypersurfaces (which we will refer to as genuine hyperbolic hyperplanes); when $\lambda = 1$ these are \emph{horospheres} (which are, intuitively, spheres of infinite radius); and finally, when $\lambda \in (0,1)$ these are \emph{equidistants}, i.e.\ connected components of the set of points having distance ${\rm artanh}\,\lambda$ from a fixed genuine hyperplane. The latter three cases are all unbounded hypersurfaces. This suggests considering totally umbilic hypersurfaces of normal curvature $\lambda \in [0,1]$ as generalized hyperplanes in hyperbolic space. The following terminology comes from \cite{Sol}.

\begin{definition}
For $\lambda \in [0,1]$, a $\lambda$-geodesic hyperplane in hyperbolic space is a complete totally umbilic hypersurface of normal curvature $\lambda$. We denote by $\Hyp_\lambda$ the set of $\lambda$-geodesic hyperplanes in $\HH^d$. In particular, $\Hyp=\Hyp_0$.
\end{definition}

Note that for $\lambda > 0$, any $\lambda$-geodesic hyperplane admits a natural choice of a unit normal vector, corresponding to the choice of a positive sign of the normal curvature. This may also be seen by the fact that exactly one of the two domains bounded by $H$ is convex, and we call this domain the convex side of $H$ (in the case $\lambda = 1$, it is known as the horoball bounded by the horosphere $H$). {In these terms, the positive choice of unit normal vector is the one pointing into the convex side of $H$.}

The space $\Hyp_\lambda$, for $\lambda \in (0,1]$, carries an isometry-invariant measure $\Lambda_\lambda$ (which is unique up to a constant), described as follows, see {\cite[Proposition 2.2]{Sol} and the discussion below it}. Again, we fix an origin $\bo\in \HH^d$ and parametrize an element $H\in\Hyp_\lambda$ by a pair $(s, u) \in \RR \times\SS^{d-1}$, where $s\in\RR$ is the signed distance from $H$ to $\bo$ ({with $s > 0$ if $\bo$ lies on the convex side of $H$, and negative otherwise}), and $u\in\SS^{d-1}$ is the unit vector (again in the tangent space $T_\bo\HH^d$) along the geodesic passing through $\bo$ and intersecting $H$ orthogonally, {while pointing outside of the convex side}. The invariant measure $\Lambda_\lambda$ on $\Hyp_\lambda$ is then defined by the relation ({we note that our normalization for $\Lambda_\lambda$ differs from that of \cite{Sol}; the latter is given by $\omega_d$ times ours})
\begin{equation}\label{eq:dHlambda}
	\int_{\Hyp_\lambda}f(H)\,\Lambda_\lambda(\dint H) = \int_{\RR}\int_{\SS^{d-1}} f(H(s,u))\,(\cosh s-\lambda\sinh s)^{d-1} \,\dint u\,\dint s,
\end{equation}
where $f:\Hyp_\lambda\to\RR$ is a non-negative measurable function and $H(s,u)$ stands for the unique element of $\Hyp_\lambda$ parametrized by $(s,u)$ as just described. Note that in the case $\lambda=0$ 
{we recover the invariant measure considered in the previous section, that is $\Lambda_0 =  \Lambda$}.

{\begin{remark}
		Although we will make no use of it below, let us mention some simpler expressions for the measures $\Lambda_\lambda$ in the planar case. For example, when the hyperbolic plane $\HH^2$ is realized as the Poincar\'{e} disc model, $\lambda$-geodesics are realized by Euclidean lines and circular arcs making an angle $\theta$ with the boundary circle, where $\cos\theta=\lambda$. Thus a $\lambda$-geodesic $L_\lambda$ with $\lambda < 1$ is uniquely determined by its two boundary points $z_\pm \in \SS^1$, ordered so that traversing $L_\lambda$ from $z_-$ to $z_+$ agrees with its positive orientation (i.e., such that the convex side lies to the left of $L_\lambda$). This sets up a bijection
		\[
		(\SS^1 \times \SS^1) \setminus \Delta \rightarrow \Hyp_\lambda, \quad (z_-, z_+) \mapsto L_\lambda(z_-, z_+),
		\]
		where $\Delta$ denotes the diagonal (in the case $\lambda=0$  the ordering of $z_\pm$ is arbitrary, and the map $(\SS^1 \times \SS^1) \setminus \Delta \to \Hyp_0$ is a double cover). The measure $\Lambda_\lambda$ can then be represented as follows:
		\[
		\int_{\Hyp_\lambda} f(L) \,\Lambda_\lambda(\dint L) = \sqrt{1-\lambda^2} \,\iint_{(\SS^1 \times \SS^1) \setminus \Delta} f\big(L_\lambda(z_-, z_+)\big) {\dint z_- \, \dint z_+ \over |z_+-z_-|^2},
		\]
		where $\dint z_\pm$ is the normalized length measure on the circle and $|\cdot|$ denotes the Euclidean norm (see the proof of \cite[Proposition 6.1]{BenjaminiJonassonSchrammTykesson} for the case $\lambda=0$). In the case $\lambda=1$, horospheres are Euclidean circles tangent to the boundary, and thus a horosphere $L_1=L_1(z,r)$ is uniquely determined by its boundary tangency point $z\in \SS^1$ and its Euclidean radius $r \in (0,1)$. The measure $\Lambda_1$ is then given by 
		\[
		\int_{\Hyp_1} f(L) \,\Lambda_1(\dint L) = \int_{\SS^1} \int_0^1 f \big(L_1(z,r)\big) r^{-2} \,\dint r \,\dint z.
		\]
\end{remark}}

\begin{example}
	{Similarly to the case $\lambda=0$, the total invariant measure of $\lambda$-geodesic hyperplanes intersecting a hyperbolic ball of radius $R$ is obtained from \eqref{eq:dHlambda} and equals
		\begin{equation*}
		\Lambda_\lambda\big(\{H \in \Hyp_\lambda \,:\, H \cap B_R \neq \emptyset  \}\big) = \int_{-R}^R (\cosh s-\lambda \sinh s)^{d-1} \,\dint s.
		\end{equation*}
		In particular, for $\lambda = 1$ we recover the total invariant measure of horospheres intersecting the ball (cf. \cite[p. 113]{GNS}),
		\begin{equation*}
			\Lambda_1\big(\{H \in \Hyp_1 \,:\, H \cap B_R \neq \emptyset  \}\big) = \int_{-R}^R e^{-(d-1)s} \,\dint s = {2 \over d-1} \,\sinh \big[(d-1)R\big].
		\end{equation*}
		In general, the total invariant measures of hyperplanes meeting a \emph{$\lambda$-convex} body with smooth boundary in $\HH^d$ is given as a certain linear combination of its volume and mean curvature integrals of its boundary (see \cite[Theorem 1]{Sol} and the remark following it).}
\end{example}

Now, fix $\lambda\in[0,1]$. Our main object is a Poisson process $\eta_{\lambda}$ on the space $\Hyp_\lambda$ whose intensity measure is the invariant measure $\Lambda_\lambda$, together with its associated random union set
$$
U_{\lambda} := \bigcup_{H\in\eta_{\lambda}}H\,.
$$
We note that for $\lambda=0$ this model reduces to the one considered in the previous subsection, while for $\lambda>0$ it has not been considered in the existing literature, as far as we know. We are interested in the following functional:
\begin{align}
S_R^{(\lambda)} := \cH^{d-1}(U_\lambda\cap B_R^d)=\sum_{H\in\eta_{\lambda}}\cH^{d-1}(H\cap B_R^d),
\label{eq:S_R_lambda_def}
\end{align}
where $B_R^d$ stands again for the hyperbolic ball with radius $R>0$ around some arbitrary but fixed point in $\HH^d$.

Our first result addresses the expectations and variances of the random variables $S_R^{(\lambda)}$. For the variances we provide only the growth as $R \to \infty$, as the exact formulas (which will be derived during the proof) are otherwise unilluminating. We emphasize that the implied constants in the variance asymptotic depend on the $d$ and $\lambda$.

\begin{theorem}\label{thm:exp_var}
The expectation of $S_R^{(\lambda)}$ is given by
\begin{equation*}
\EE S_{R}^{(\lambda)} = \cH^{d}(B_R^d) = \omega_{d-1} \int_0^R \sinh^{d-1}(u)\,\dint u.
\end{equation*}
Additionally, the variance of $S_R^{(\lambda)}$ satisfies 
	\begin{equation*}
	\Var S_R^{(\lambda)} \asymp \begin{cases}
	e^R &: \text{for } \lambda \in [0,1) \text{ and }d = 2\\
	Re^{2R} &: \text{for } \lambda \in [0,1) \text{ and }d = 3\\
	 e^{2(d-2)R} &: \text{for } \lambda \in [0,1) \text{ and }d \geq 4\\
	 R \, e^{(d-1)R} &: \text{for } \lambda =1 \text{ and any }d\geq 2.	\end{cases}
	\end{equation*}
\end{theorem}

\begin{remark}
{Let us remark that the expectation in Theorem \ref{thm:exp_var} is independent of the parameter $\lambda$. This is due to the fact that the expected surface area of a section of a hyperbolic ball by a moving $\lambda$-geodesic hyperplane is given by the Crofton formula (see \cite[Proposition 3.1]{Sol}), whose form is independent of $\lambda$.}
\end{remark}

\begin{remark}
In all cases except for the case $d=2$ and $\lambda \in [0,1)$ the statement regarding the variance can be upgraded to asymptotic equivalence, with explicit constants (including dependence on $\lambda$). This will be clear from the proof of Theorem \ref{thm:exp_var}
\end{remark}

Our next result delivers for $\lambda<1$ first a quantitative bound on the distance between the normalized random variables $S_R^{(\lambda)}$ and a standard Gaussian random variable for $d=2$ and $d=3$. To measure these distances we use the Wasserstein and the Kolmogorov distance $d_{\Wass}$ and $d_{\Kol}$, respectively. For two random variables $X,Y$, they are defined as
$$
d_{\star}(X,Y) := \sup\big|\EE[h(X)]-\EE[h(Y)]\big|,\qquad\star\in\{\Kol,\Wass\},
$$
where the supremum is taken over all Lipschitz functions $h:\RR\to\RR$ with Lipschitz constant at most one in case of the Wasserstein distance (for $\star=\Wass$) and over all indicator functions of the form $h={\bf 1}_{(-\infty,x]}$, $x\in\RR$, for the Kolmogorov distance (for $\star=\Kol$). This generalizes {the case $i=d-1$ of }the  central limit theorem in \cite[Theorem 5]{HeroldHugThaele} to general $\lambda$-geodesic hyperplanes. Moreover, for $d\geq 4$ we shall again characterize the non-Gaussian limit distribution, generalizing thereby Theorem \ref{thm:NonClt} from the previous section.

\begin{theorem}\label{thm:NonCltlambda}
	Fix $\lambda\in[0,1)$.
	\begin{itemize}
	\item[(i)] Let $N$ be a standard Gaussian random variable. Then there exist an absolute constant $c\in(0,\infty)$ and a constant $R_\lambda \in(0,\infty)$ only depending on $\lambda$ such that for $R > R_\lambda$
	$$
	d_\star\Bigg({S_R^{(\lambda)}-\EE S_R^{(\lambda)}\over\sqrt{\Var S_R^{(\lambda)}}},N\Bigg) \leq \begin{cases}
c\,(1-\lambda)^{-1}\,e^{-R/2} &: \text{for }d=2,\\
c\,(1-\lambda)^{-1/2} \,R^{-1} &: \text{for }d=3,
			\end{cases}
	$$
	where $\star\in\{\Kol,\Wass\}$.
	\item[(ii)]
	Suppose that $d\geq 4$.
Then,
	$$
	{S_R^{(\lambda)}-\EE S_R^{(\lambda)}\over e^{R(d-2)}} \overset{D}{\longrightarrow} \frac{\omega_{d-1}} {(d-2)2^{d-2}\sqrt{1-\lambda^2}} \, Z_{d,\lambda},\qquad\text{as}\;\, R\to\infty,
	$$
where $Z_{d,\lambda}$ is an infinitely divisible, zero-mean random variable defined by
$$
Z_{d,\lambda} := \lim_{T\to+\infty} \Bigg(\sum_{\substack{s\in\zeta_{d,\lambda}\\ s\leq T}} \cosh^{-(d-2)}(s) - 2(1-\lambda^2)^{\frac{d-1}{2}}  \sinh T\Bigg)
\quad
\text{ in $L^2$ and a.s.,}
$$
and $\zeta_{d,\lambda}$ is a Poisson process on $[0,\infty)$ with density $s\mapsto 2(1-\lambda^2)^{\frac{d-1}{2}} \cosh^{d-1}(s)$.
\end{itemize}
\end{theorem}
The characteristic function of the random variable $Z_{d,\lambda}$ appearing in Theorem \ref{thm:NonCltlambda} (ii) is the ${(1-\lambda^2)^{\frac{d-1}{2}} }$-th power of the characteristic function of $Z_d$ from Theorem~\eqref{thm:NonClt}. This means that both variables can be embedded into the same L\'evy process (and belong to the same convolution semigroup). The properties of $Z_{d,\lambda}$ are similar to those of $Z_d$.

We shall now discuss the remaining case $\lambda=1$, i.e., the case of the total surface area of a Poisson process of horospheres in $\HH^d$. What distinguishes $\lambda$-geodesic hyperplanes for $\lambda<1$ from horospheres (that is, $\lambda$-geodesic hyperplanes for $\lambda=1$) is their intrinsic geometry. In fact, for $\lambda<1$ the intrinsic geometry is again hyperbolic with constant sectional curvature $-(1-\lambda^2)$, while for $\lambda=1$ the intrinsic geometry is Euclidean, see Section \ref{sec:background} below. This is also reflected by the fluctuation behaviour as we shall see now. In fact, we have Gaussian fluctuations in any space dimension, but surprisingly with the non-standard variance $1/2$ instead of $1$. The result reads as follows.

\begin{theorem}\label{thm:CLT_horosphere}
	Let $N_{1\over 2}$ be a centred Gaussian random variable of variance $\frac{1}{2}$. Then
	\begin{equation*}
	\frac{S_R^{(1)} - \EE S_R^{(1)}}{\sqrt{\Var S_R^{(1)}}} \xrightarrow{D} N_{1\over 2}, \qquad \text{as } R \to \infty.
	\end{equation*}	
\end{theorem}

\section{Proof of Theorem \ref{thm:NonClt}}\label{sec:ProofHyperplanes}

For $R>0$, define $f_R:\Hyp\to[0,\infty)$ by $f_R(H):=\cH^{d-1}(H\cap B_R^d)$. Then $S_R$ can be represented as $S_R=\sum_{H\in\eta}f_R(H)$. To describe the characteristic function of the random variable $S_R$, we first observe that $f_R(H)$ only depends on the hyperbolic distance $d_h(H,p)$ from $H$ to the centre $p$ of the ball $B_R^d$. This allows us to write $f_R(s)$ instead of $f_R(H)$ for any $H\in\Hyp$ with $s=d_h(H,p)$. Moreover, the point process of the distances of the hyperplanes from $\eta$ is an inhomogeneous Poisson process $\zeta_d$ on $[0,\infty)$ with density function $s\mapsto {2}\cosh^{d-1}(s)$. This follows directly from the concrete representation of the invariant measure $\Lambda$ on $\Hyp$, see \eqref{eq:dH}. As a consequence, if $\mathfrak{i}=\sqrt{-1}$ denotes the imaginary unit, we have from \cite[Lemma 15.2]{KallenbergVol1} that, for $t\in\RR$,
\begin{align*}
	\EE e^{\mathfrak{i}tS_R} = \exp\Big({2}\int_0^R(e^{\mathfrak{i}tf_R(s)}-1)\cosh^{d-1}(s)\,\dint s\Big)
\end{align*}
and hence
\begin{align*}
	\psi_R(t):=\EE e^{\mathfrak{i}t{S_R-\EE S_R\over e^{R(d-2)}}} = \exp\Big({2}\int_0^R(e^{\mathfrak{i}tg_R(s)}-1-{\mathfrak{i}tg_R(s)})\cosh^{d-1}(s)\,\dint s\Big)
\end{align*}
with $g_R(s):=f_R(s)/e^{R(d-2)}$. In the next lemma, which {complements} \cite[Lemma 7]{HeroldHugThaele}, we determine the asymptotic behaviour of $g_R(s)$, as $R\to\infty$.


\begin{lemma}\label{lem:constantC}
Let $s\in[0,\infty)$. Then $\lim_{R\to\infty} g_R(s) =  g(s) :=  {\omega_{d-1}\over (d-2)2^{d-2}}\,\cosh^{-(d-2)}(s)$.
\end{lemma}
\begin{proof}
According to \cite[Theorem 3.5.3]{Ratcliffe}, it holds that
$$
g_R(s)=e^{-(d-2)R}\omega_{d-1}\int_0^{\arcosh({\cosh(R)\over\cosh(s)})}\sinh^{d-2}(u)\,\dint u.
$$
Next, we recall the logarithmic representation of $\arcosh(t)$:
$$
\arcosh(t) = \log(t+\sqrt{t^2-1}) = \log(2t)+o(1),\qquad\text{as}\ t\to\infty,
$$
where $o(1)$ stands for a sequence which converges to zero, as $t\to\infty$. Thus, as $R\to\infty$,
\begin{align*}
\arcosh\Big({\cosh(R)\over\cosh(s)}\Big) = \log(2 \cosh R) - \log(\cosh s) + o(1)
											   = R - \log (\cosh s) + o(1).
\end{align*}
Next, we observe that
\begin{align}
\int_0^t\sinh^{d-2}(u)\,\dint u \sim \int_0^t \Big({e^u\over 2}\Big)^{d-2}\,\dint u \sim {1\over (d-2)2^{d-2}}e^{(d-2)t},\label{eq:hyper_int_asymp}
\end{align}
as $t\to\infty$. Combining this with the above representation for $g_R(s)$ we arrive at
\begin{align*}
g_R(s) &\sim e^{-(d-2)R}\omega_{d-1}\cdot{1\over (d-2)2^{d-2}}\cdot e^{(d-2)(R-\log(\cosh(s))+o(1))}\\
&={\omega_{d-1}\over (d-2)2^{d-2}}\,e^{(d-2)(-\log(\cosh s)+o(1))}\\
& \xrightarrow[R \to \infty]{} {\omega_{d-1}\over (d-2)2^{d-2}}\,\cosh^{-(d-2)}(s).
\end{align*}
This completes the proof.
\end{proof}

Lemma \ref{lem:constantC} shows that pointwise in $[0,\infty)$ the function $g_R$ converges to the function $g(s)={\omega_{d-1}\over (d-2)2^{d-2}}\,\cosh^{-(d-2)}(s)$, as $R\to\infty$. In order to prove that this together with the dominated convergence theorem implies that
\begin{align}\label{eq:ConvCharFkt}
\lim_{R\to\infty}\psi_R(t) = \psi(t) := \exp\Big({2}\int_0^\infty(e^{\mathfrak{i}tg(s)}-1-\mathfrak{i}tg(s))\cosh^{d-1}(s)\,\dint s\Big),
\end{align}
it remains to find an integrable upper bound for {the absolute value} of the integrand $(e^{\mathfrak{i}tg_R(s)}-1-{\mathfrak{i}tg_R(s)})\cosh^{d-1}(s)$. {Using Taylor expansion (see, e.g., \cite[Lemma 6.15]{KallenbergVol1}) gives} for any $s\geq 0$,
$$
|e^{\mathfrak{i}tg_R(s)}-1-{\mathfrak{i}tg_R(s)}| \leq {1\over 2}t^2g_R(s)^2
$$
and since $\cosh(s) \leq e^s$, we see that
$$
|(e^{\mathfrak{i}tg_R(s)}-1-{\mathfrak{i}tg_R(s)})\cosh^{d-1}(s)| \leq {1\over 2}t^2 g_R(s)^2e^{(d-1)s}.
$$
Also, we have that $g_R(s)\leq{\omega_{d-1}\over d-2}e^{-s(d-2)}$ by \cite[Lemma 7]{HeroldHugThaele}. This leads to the upper bound
$$
|(e^{\mathfrak{i}tg_R(s)}-1-{\mathfrak{i}tg_R(s)})\cosh^{d-1}(s)| \leq {\omega_{d-1}^2\over 2(d-2)^2}t^2e^{-s(d-3)},
$$
which is in fact integrable on $[0,\infty)$ by our assumption that $d\geq 4$. So, \eqref{eq:ConvCharFkt} proves that, as $R\to\infty$, the random variables ${S_R-\EE S_R\over e^{R(d-2)}}$ converge in distribution to a random variable with characteristic function $\psi$. To conclude, we have to identify this random variable as $C_dZ_d$, where $C_d := \frac{\omega_{d-1}}{(d-2)2^{d-2}}$. This can be done with the help of the L\'evy-Khinchin formula~\cite[Theorem~8.1]{Sato_book}, or more directly seen as follows.
We have to show that the characteristic function of the random variable $Z_d$ is $\psi(C_d^{-1}t)$. Denoting $h(s) := C_d^{-1}g(s )= \cosh^{-(d-2)}(s)$, we have
\begin{equation}\label{eq:psi_rescaled}
\psi(C_d^{-1}t) =  \exp\Big({2}\int_0^\infty(e^{\mathfrak{i}th(s)}-1-\mathfrak{i}th(s))\cosh^{d-1}(s)\,\dint s\Big).
\end{equation}
Recall the random variables $Y_T$, defined by $Y_T = \sum_{s \in \zeta_d, s \leq T}h(s)$, where $\zeta_d$ is an inhomogeneous Poisson process on $[0,\infty)$ with density function $s\mapsto{2}\cosh^{d-1}(s)$. As explained in Remark \ref{rem:Zd}, as $T \to \infty$ the random variables $Y_T-\EE Y_T$ converge a.s.\ and in $L^2$, and their limit is by definition $Z_d$. In particular, $Y_T-\EE Y_T$ converges to $Z_d$ in distribution. The characteristic function of the random variable $Y_T-\EE Y_T$ is computed easily as
\begin{equation*}
 \EE e^{it(Y_T - \EE Y_T)} = \exp\Big({2}\int_0^T(e^{\mathfrak{i}th(s)}-1-\mathfrak{i}th(s))\cosh^{d-1}(s)\,\dint s\Big),
\end{equation*}
see, e.g.,\ \cite[Lemma 7.1]{KallenbergVol1}. Using a dominated convergence argument as above one sees that as $T \to \infty$, these characteristic functions converge to \eqref{eq:psi_rescaled}, and so $Z_d = \lim_{T \to \infty}(Y_T -\EE Y_T)$ indeed has characteristic function $\psi(C_d^{-1}t)$, as required. This completes the proof.
\qed

\section{Background material on \texorpdfstring{$\lambda$}{lambda}-geodesic hyperplanes}\label{sec:background}

In this section we collect some facts about the geometry of $\lambda$-geodesic hyperplanes which will be needed in the proofs of Theorem \ref{thm:NonCltlambda} and Theorem \ref{thm:CLT_horosphere}.  We begin by  introducing some notation which is relevant for the case $\lambda \in [0,1)$. Let $\theta \in (0, \frac{\pi}{2})$ be an angle such that $\cos \theta = \lambda $. We write also
\begin{equation}
\mu := \sin \theta = \sqrt{1-\lambda^2} ; \quad m := \tan\theta = \frac{\sqrt{1-\lambda^2}}{\lambda}.
\end{equation}
Finally, we define
\begin{equation*}
\Delta := {\rm artanh}\, \lambda = \frac{1}{2} \log \frac{1 + \lambda}{1-\lambda} = - \log \tan\frac{\theta}{2}  \in (0, \infty).
\end{equation*}
We recall from Section \ref{subsec:lambdaHyperplanes}  that $\Delta$ is the distance from a $\lambda$-geodesic hyperplane (with $\lambda \in (0,1)$) to the genuine geodesic hyperplane from which it is equidistant.  We note also that the expression for the invariant measure $\Lambda_\lambda$ from \eqref{eq:dHlambda} may be simplified using the identity
\begin{equation}\label{eq:density_Delta}
\cosh s-\lambda\sinh s = \mu \cosh(s-\Delta),
\end{equation}
which will be used frequently in the sequel and follows from the formula $\cosh(s-\Delta) = \cosh s \cosh \Delta - \sinh s \sinh \Delta$ together with $\cosh \Delta = 1/\sqrt {1-\lambda^2}$ and $\sinh \Delta = \lambda / \sqrt{1-\lambda^2}$.

\subsection{\texorpdfstring{$\lambda$}{Lambda}-geodesic hyperplanes in the upper half-space model}\label{subsec:upper_half_plane}

Below we will perform computations in the upper half-space model for hyperbolic space:
\[
\HH^d = \{ (x_1, \ldots, x_d)\in\RR^d \,:\, x_1 > 0 \},
\]
({note that we take the first coordinate to be positive, contrary to standard practices}) equipped with the hyperbolic Riemannian metric
\begin{equation}
\dint s^2 = \frac{\dint x_1^2 + \cdots +  \dint x_d^2}{x_1^2}.
\end{equation}
{
The induced distance function on $\HH^d$ is given by
\[
d(x,y) = \arcosh \left[1 + { \sum_{j=1}^d (x_j-y_j)^2 \over 2 x_1 y_1}\right].
\]
We recall the fact that in this model metric balls are also Euclidean balls, albeit with different centers and radii (see, e.g.,  Fact 1 in \cite[Section 13]{CannonFloydKenyonParry}).  In fact, using the special case of the formula above for the hyperbolic distance for points lying on the same line parallel to the $x_1$-axis, namely
\[
d\bigl((a, x_2, \ldots, x_d),(b,x_2, \ldots, x_d)\bigr) = \left|\log {a \over b}\right|,
\] 
it follows that the hyperbolic radius of a hyperbolic ball (realized in the model as a Euclidean ball) is given by one-half of the logarithm of the ratio between the maximal and minimal $x_1$-coordinates of points in the ball. For example, the hyperbolic ball $B_R^d = B(e_1, R)$ of radius $R$ around the point $e_1=(1,0,\ldots, 0)$ is realized as the Euclidean ball with center $\cosh R\cdot e_1$ and radius $\sinh R$.
}

\renewcommand{\thefootnote}{\arabic{footnote}}
\setcounter{footnote}{0}
Next we recall (see \cite[\S 8.5]{doC} {or \cite[Theorem 7.29]{Spivak}, see also \cite[Proposition 1.2]{Cza}})\footnotemark{} that in this model, $\lambda$-geodesic hyperplanes are described as follows:
\begin{enumerate}[(i)]
	\item When $\lambda \in [0,1)$, they are given by (the intersection of the upper half-space with) Euclidean hyperplanes and spheres which intersect the boundary of $\HH^d$ at an angle $\theta$ as above (i.e. with $\cos \theta = \lambda$).
	\item When $\lambda = 1$, they are horospheres, that is Euclidean spheres tangent to the boundary of $\HH^d$ or Euclidean hyperplanes parallel to the boundary.
\end{enumerate}
\footnotetext{In fact, the listed references do not contain a complete proof of the above description. In \cite{Spivak} the cases $\lambda=0$ and $\lambda=1$ are proven in full, while for $\lambda\in(0,1)$ the precise relation $\lambda = \cos\theta$ is missing (see Theorem 7.29 and the discussion on pages 78--79); the same relation is only given as Exercise 6(e) of \cite[\S 8.5]{doC}, while in \cite{Cza} its proof relies on a reference unavailable in English. For completeness we sketch the proof of this relation. By \cite[Theorem 7.29]{Spivak} a $\lambda$-geodesic hyperplane with $\lambda \in (0,1)$ is a Euclidean hyperplane or sphere meeting the boundary at some angle $\theta \in (0, {\pi \over 2})$. By transitivity we may assume that it is a hyperplane of the form $\{x_1 = \tan\theta \cdot  x_d\}$. Since it is umbilical it suffices to compute the curvature of its intersection with the plane spanned by the $x_1$- and  $x_d$-coordinates and show that it is equal to $\cos\theta$. Finally, this two-dimensional computation is carried out in \cite[Lemma 3]{GR}.}

Of particular use to us below will be the `linear' $\lambda$-geodesic hyperplanes, i.e., those realized in this model as Euclidean hyperplanes. We note that for such a $\lambda$-geodesic hyperplane $H$, its convex side is (the intersection with $\HH^d$ of) the Euclidean half-space lying above $H$, {in the sense that the normal vector to $H$ pointing into this half-space has positive first coordinate}. 

\subsection{Intrinsic geometry of \texorpdfstring{$\lambda$}{lambda}-geodesic hyperplanes}

An important property of $\lambda$-geodesic hyperplanes is that they are themselves manifolds of constant sectional curvature $-(1-\lambda^2)$. In particular, their intrinsic geometry is Euclidean for $\lambda = 1$, and (up to rescaling) hyperbolic for $\lambda \in [0,1)$. This can be seen abstractly using the Gauss equation (see, e.g., \cite[Theorem 8.5]{Lee} or \cite[Theorem 6.2.5]{doC}), but can also be checked explicitly in our model, as we do next.

We start with the case of horospheres. By transitivity of the isometry group it is enough to consider linear horospheres of the form
\begin{equation*}
H = \{x=(x_1,\ldots,x_d) \in \HH^d\,:\, x_1 = c  \}
\end{equation*}
for some $c>0$. The induced Riemannian metric on $H$ is then $\dint s^2\bigl|_H = c^{-2} (\dint x_2^2 + \cdots + \dint x_d^2)$, which is Euclidean.

In the case of $\lambda \in [0,1)$, it suffices, again by symmetry, to consider linear $\lambda$-geodesic hyperplanes of the form
\begin{equation}\label{eq:slanted_hypersphere}
H = \{x \in \HH^d \,:\, x_1 = m(x_d - a) \}
\end{equation}
for some $a\in\RR$. {Since $H$ makes an angle $\theta$ with the boundary hyperplane, we have $m = \tan \theta$ (note that this implies $1 + m^{-2} = \mu^{-2}$)}.  The induced Riemannian metric on $H$ is
\begin{align*}
\dint s^2 \bigr|_{H} &= \frac{\dint x_1^2 + \cdots + \dint x_{d-1}^2 +  \frac{1}{m^2} \dint x_1^2}{x_1^2} = \frac{ (1 + m^{-2})\dint x_1^2 + \dint x_2^2+ \cdots+ \dint x_{d-1}^2}{x_1^2}.
\end{align*}
Now if we define new coordinates on $H$ by
\begin{equation*}
v_1 = \frac{1}{\mu} x_1, \, v_2 = x_2,\, \ldots,\, v_{d-1} = x_{d-1},
\end{equation*}
we get that
\begin{equation}\label{eq:coords_hyperbolic}
\dint s^2 \bigr|_H = \frac{1}{\mu^2} \frac{\dint v_1^2 + \cdots + \dint v_{d-1}^2}{v_1^2},
\end{equation}
which is indeed a rescaled hyperbolic metric. Let us call coordinates such as \eqref{eq:coords_hyperbolic} {intrinsic hyperbolic coordinates} on $H$.

\subsection{\texorpdfstring{$\lambda$}{Lambda}-geodesic sections of the ball}

Let $B_R^d$ be a hyperbolic ball of radius $R$. We take the centre of the ball as our origin $\bo \in \HH^d$, and compute the  $(d-1)$-volume of the intersection of $B_R^d$ with a $\lambda$-geodesic hyperplane which has oriented distance $s$ to $\bo$. Observe that this is well-defined, as hyperbolic rotations about $\bo$ are transitive on such hyperplanes and preserve $B_R^d$. Let us denote such a $\lambda$-geodesic hyperplane by $H(s)$.
\begin{proposition}\label{prop:intersection}
	Fix $\lambda \in [0,1]$,  and let $s \in \RR$. Then the intersection $H(s) \cap B_R^d$ is an intrinsic ball in $H(s)$, which is non-empty if and only if $|s| \leq R$. Its volume is given by the following formulas.
	\begin{enumerate}
		\item 	If $\lambda = 1$, then
		\begin{equation}\label{eq:vol_intersection_horo}
		\cH^{d-1}(H(s) \cap B_R^d) = \kappa_{d-1} \bigl[ 2 e^s (\cosh R - \cosh s)\bigr]^{\frac{d-1}{2}}.
		\end{equation}
		\item If $0\leq \lambda < 1$, then
		\begin{equation}\label{eq:vol_intersection}
		\cH^{d-1} (H(s) \cap B_R^d) = \frac{\omega_{d-1}}{\mu^{d-1}} \int_0^{\rho(s;R)} \sinh^{d-2}u \, \dint u,
		\end{equation}
		where
		\begin{equation}\label{eq:rho}
		\rho(s;R) 
				=\arcosh \frac{ \cosh R - \sinh \Delta  \sinh(s - \Delta)}{\cosh \Delta \cosh(s-\Delta)}.
		\end{equation}
	\end{enumerate}
\end{proposition}
In the case $\lambda \in [0,1)$, we will need more amenable bounds on the radius $\rho(s;R)$ appearing in Formula \eqref{eq:vol_intersection} above. These are given by the following result.

\begin{lemma}\label{lem:bound_rho}
	Let $\lambda \in [0,1)$ and $s\in [-R, R]$.
	Then one has the following bounds on $\rho(s;R)$:
		\begin{align}
	{\frac{\cosh (R-\Delta)}{\cosh (s-\Delta)}} & {\leq  \cosh\rho(s;R) \leq \frac{\cosh (R+\Delta)}{\cosh (s-\Delta)} }\label{eq:bound_rho_arcosh} \\
		\intertext{and}
		R  -  \Delta - |s-\Delta| & \leq \rho(s;R) \leq R + \Delta - |s-\Delta|+ \log 2. \label{eq:bound_rho}
		\end{align}
\end{lemma}

Finally, when $\lambda \in [0,1)$  and $d > 2$, we will additionally need the following result about the asymptotic behaviour of the intersection volume.

\begin{lemma}\label{lem:intersection_asymp}
Let $\lambda \in [0,1)$ and $d > 2$. Then for any $s \in [-R,R]$ one has
  \begin{equation}\label{eq:inter_vol_bound}
\cH^{d-1}(H(s) \cap B_R^d) \leq \frac{\omega_{d-1}}{\mu^{d-1}} \frac{1}{d-2}  \left[\frac{\cosh (R+\Delta)}{\cosh(s-\Delta)}\right]^{d-2}.
\end{equation}
Moreover, for fixed $s \in \RR$ as $R \to \infty$,
\begin{equation}\label{eq:inter_vol_asymp}
\cH^{d-1}(H(s) \cap B_R^d) \sim \frac{\omega_{d-1}}{(d-2)2^{d-2}} \cdot   \mu^{-1} e^{(d-2)R} \cosh^{-(d-2)}(s-\Delta).
\end{equation}
\end{lemma}

Since the proof of Proposition \ref{prop:intersection} differs between the cases $\lambda < 1$ and $\lambda = 1$, we prove the two cases separately. We start with the proof of the simpler case $\lambda = 1$.
\begin{proof}[Proof of Proposition \ref{prop:intersection} for $\lambda = 1$]
We work, as above, in the upper half-space  model. We consider the family of linear horospheres given by $\{x_1 \equiv c\}$ with $c>0$. Such a  horosphere has distance $|\log c|$ to  the origin $\bo = e_1 = (1,0,\ldots, 0)$, and the corresponding horoball is given by $\{x_1 \geq c \}$. Therefore, in terms of the oriented distance $s \in \RR$ we have
	\begin{equation*}
	H(s) = \{x_1 = e^{-s}  \}.
	\end{equation*}
	Recall also that the induced Riemannian metric on $H(s)$ is conformally Euclidean, namely
	\begin{equation*}
	\dint s^2 \bigr|_{H(s)} = e^{2s} \bigl( \dint x_2^2 + \cdots + \dint x_d^2\bigr).
	\end{equation*}
	In particular, the volume of the intersection $H(s) \cap B_R^d$ is simply $e^{(d-1)s}$ times its Euclidean volume (in the coordinates $x_2, \ldots, x_d$). {Recall from Section \ref{subsec:upper_half_plane} that} in our model the ball $B_R^d = B(\bo, R)$ is realized as the Euclidean ball with center $\cosh R\cdot e_1$ and radius $\sinh R$. The intersection $H(s) \cap B_R^d$ is therefore a Euclidean ball of radius $\rho$, where (see the left panel of Figure \ref{fig_horosphere_section})
	\begin{equation*}
	\rho^2 = \sinh^2 R - (\cosh R - e^{-s})^2 = 2e^{-s} (\cosh R - \cosh s).
	\end{equation*}
	\begin{figure}[t]
	\begin{center}
	\includegraphics[width=0.45\columnwidth]{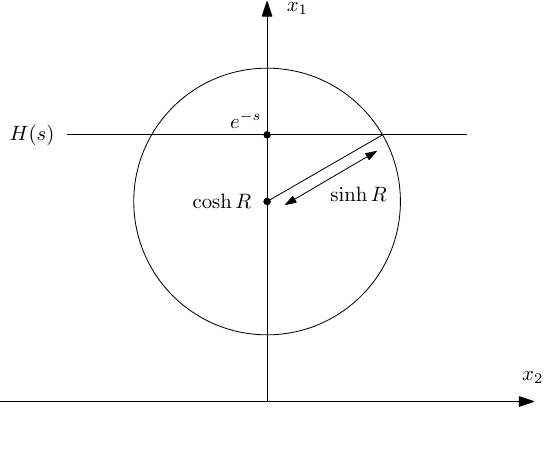}\qquad
	\includegraphics[width=0.45\columnwidth]{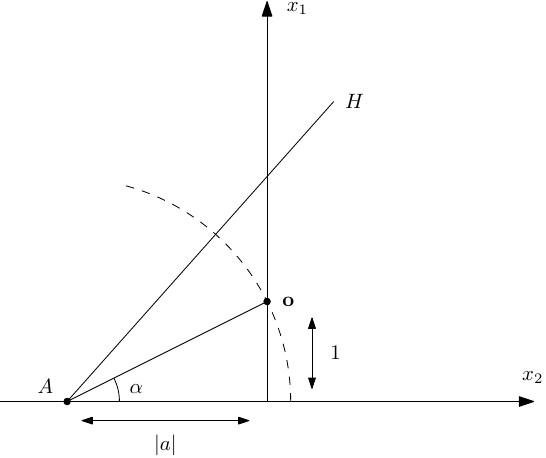}
	\caption{Left panel: Horospheric section of the ball. Right panel: A linear equidistant.}\label{fig_horosphere_section}
	\end{center}
	\end{figure}
	Evidently, the intersection is non-empty if and only if $|s| \leq R$. Finally, recalling the conformal factor, we compute the volume of the intersection as
	\begin{align*}
	\cH^{d-1}(H(s) \cap B_R^d) &= e^{(d-1)s} \cdot \kappa_{d-1} \left[2e^{-s} (\cosh R - \cosh s) \right]^{\frac{d-1}{2}}  \\
	&= \kappa_{d-1} \left[2e^{s} (\cosh R - \cosh s) \right]^{\frac{d-1}{2}},
	\end{align*}
	as asserted.
\end{proof}

Next we consider the case $\lambda < 1$.  Again it suffices to consider linear $\lambda$-geodesic hyperplanes of the form
 \begin{equation*}
H = \{x \in \HH^d \,:\, x_1 = m(x_d - a) \}.
\end{equation*}
First, let us compute the translation parameter $a$ in terms of the signed distance $s$  to $\bo = e_1 = (1,0,\ldots, 0)$ for that situation.

\begin{lemma}\label{lem:a}
	In the setting above one has $a = \sinh(s-\Delta)$.
\end{lemma}
\begin{proof}
	Note that the problem takes place entirely in the $(x_1,x_d)$ plane, so we can suppose that $d=2$, so that
	\begin{equation*}
	H = \{(x_1, x_2) \,:\, x_1 = m(x_2-a) \}.
	\end{equation*}
	Denote by $A = (0,a)$ the intersection point of $H$ with the $x_2$-axis, and by $\alpha {\in [0, {\pi \over 2}]} $ the {acute} angle between the line $\overline{A \bo}$ and the $x_2$-axis (see the right panel of  Figure \ref{fig_horosphere_section}).
 Clearly, the distance between $\bo$ and $H$ is attained at the geodesic represented by a semi-circle with centre $A$, and hence by a standard computation (see, e.g., \cite[Figure 24]{CannonFloydKenyonParry}) for hyperbolic distances,  one has
	\[
	|s| = \left| \log \frac{\tan\frac{\theta}{2}}{\tan \frac{\alpha}{2}} \right|.
	\]
{Recall from Section \ref{subsec:upper_half_plane} that $s >0 $ when $\bo$ lies on the convex side of $H$, which holds precisely when $\alpha > \theta$, and $s \leq 0$ otherwise. With this we deduce that}
	\[
	\tan\frac{\alpha}{2} = e^s \tan\frac{\theta}{2} = e^{s-\Delta}.
	\]
	Finally, one has $\tan \alpha = -\frac{1}{a}$   and hence
	\[
	a = -\frac{1}{\tan\alpha} = \frac{\tan^{2}\frac{\alpha}{2} - 1}{ 2 \tan \frac{\alpha}{2}} = \sinh(s-\Delta).
	\]
	This completes the proof of the lemma.
\end{proof}

\begin{proof}[Proof of Proposition \ref{prop:intersection} for $\lambda <1$]
	We work again in the upper half-space model. Let $H(s)$ be any $\lambda$-geodesic hyperplane with signed distance $s$ from $\bo$. By the previous lemma, we can take $H = \{x_1 = m(x_d-a)\}$, where $a = \sinh(s - \Delta)$. On the other hand, the ball $B_R^d$ is in our model the Euclidean ball with centre $\cosh R \cdot e_1$ and radius $\sinh R$. Then  the intersection $B_R^d\cap  H$ is given by
	$$
	(x_1 - \cosh R)^2 + x_2^2 + \ldots + x_{d-1}^2 + \left(\frac{x_1}{m} + a\right)^2 \leq \sinh^2 R,
	$$
	{which simplifies to}
	$$
	(1 + m^{-2}) x_1^2 - 2 \left(\cosh R - \frac{a}{m}\right) x_1 + x_2^2 + \ldots + x_{d-1}^2 + a^2 + 1 \leq 0.
	$$	
	{Finally, using intrinsic hyperbolic coordinates \eqref{eq:coords_hyperbolic} we get}
	\begin{equation}\label{eq:quadratic}
	v_1^2 - 2 \left(\mu \cosh R - \frac{\mu a}{m}\right) v_1 + v_2^2 + \ldots + v_{d-1}^2  + a^2 + 1  \leq 0.
	\end{equation}
	We note that this defines a Euclidean ball, and hence also a hyperbolic ball, in these coordinates. {As noted in Section \ref{subsec:upper_half_plane}}, the hyperbolic radius of such a ball, which we denote by $\rho(s;R)$, is given by one-half of the logarithm of the ratio between the maximal and minimal $v_1$-coordinates. That is,
	\begin{equation*}
	\rho(s;R) = \frac{1}{2} \log \frac{v_1^+}{v_1^-}.
	\end{equation*}
To find $v_1^\pm$, we plug $v_2 = \cdots = v_{d-1} = 0$ in \eqref{eq:quadratic} and solve the quadratic equation. This leads to the following equation, after plugging the value $a = \sinh (s - \Delta)$ and using $\frac{\mu}{m} = \lambda$:
	\[
	v_1^2 - 2 (\mu \cosh R - \lambda \sinh(s-\Delta)) v_1 + \cosh^2(s-\Delta) = 0.
	\]
	The two roots of the above equation are
	\[
	v_1^{\pm} = \biggl(\mu \cosh R - \lambda \sinh(s-\Delta)\biggr)  \left[ 1 \pm \sqrt{1 - \left(\frac{\cosh(s-\Delta)}{\mu \cosh R - \lambda \sinh(s-\Delta) }\right)^2 } \,\right]
	\]
	{(observe that the fraction inside the brackets is at most $1$, since
	\begin{align*}
	\cosh(s-\Delta) + \lambda \sinh(s-\Delta) & = \mu \big[ \cosh \Delta \cosh(s-\Delta) + \sinh \Delta \sinh(s-\Delta) \big] = \mu \cosh s \\
		& \leq \mu \cosh R,
	\end{align*}
	where we have used $\cosh \Delta = {1 \over \mu}$ and $\sinh \Delta = {\lambda \over \mu}$)}, and their ratio gives
	\begin{equation}\label{eq:rho_ugly}
		\rho(s;R) =  \frac{1}{2} \log \frac{1 + \sqrt{1-\left(\frac{\cosh(s-\Delta)}{\mu \cosh R - \lambda \sinh(s - \Delta)} \right)^2}}{1 - \sqrt{1-\left(\frac{\cosh(s-\Delta)}{\mu \cosh R - \lambda \sinh(s - \Delta)} \right)^2}}.
	\end{equation}
	To simplify the latter expression, denote
	\[
	q = \frac{\mu \cosh R - \lambda \sinh(s - \Delta)} {\cosh(s-\Delta)},
	\]
	so that \eqref{eq:rho_ugly} becomes
	\begin{equation}\label{eq:rho_q}
	 \rho(s;R) = \frac{1}{2} \log \frac{1+\sqrt{1-q^{-2}} }{1-\sqrt{1-q^{-2}}} = \log (q + \sqrt{q^2-1}) = \arcosh (q),
	\end{equation}
	where in the last equality we again used the logarithmic representation for the inverse hyperbolic cosine. Finally, using that
	\begin{equation*}
	\lambda = \tanh \Delta\qquad\text{and}\qquad \mu = \sqrt{1-\tanh^2 \Delta} = \frac{1}{\cosh \Delta},
	\end{equation*}
	we get
	\begin{equation*}
	q = \frac{\mu\cosh R - \lambda \sinh(s-\Delta)}{\cosh(s-\Delta)} =  \frac{\cosh R - \sinh \Delta \sinh(s-\Delta)}{ \cosh \Delta \cosh(s-\Delta)}.
	\end{equation*}
	Together with \eqref{eq:rho_q}, this gives the formula \eqref{eq:rho} for the hyperbolic radius $\rho(s;R)$. Finally, since the intrinsic metric is rescaled hyperbolic (recall \eqref{eq:coords_hyperbolic}), the standard formula for the volume of the hyperbolic ball (see, e.g., \cite[Theorem 3.5.3]{Ratcliffe}) gives \eqref{eq:vol_intersection}.
\end{proof}

Next, we prove the estimates for the radius $\rho$ in the case $\lambda \in [0,1)$.

\begin{proof}[Proof of Lemma  \ref{lem:bound_rho}]
	For $s \in [-R, R]$ we have $\sinh(s-\Delta) \in [-\sinh(R+\Delta), \sinh(R-\Delta)]$. This gives, using the hyperbolic identity $\cosh(x-y) = \cosh x \cosh y - \sinh x \sinh y$,
	\begin{align*}
	{\cosh R - \sinh \Delta \sinh(s-\Delta)} &\leq \cosh (R+\Delta - \Delta) + \sinh \Delta \sinh(R+\Delta) \\
	&= \cosh(R+\Delta) \cosh \Delta.
	\end{align*}
This implies the upper bound
\begin{equation}\label{eq:rho_upper_bd}
{	\cosh\rho(s;R) =  \frac{ \cosh R - \sinh \Delta  \sinh(s - \Delta)}{\cosh \Delta \cosh(s-\Delta) } \leq  \frac{\cosh(R + \Delta)}{\cosh(s-\Delta)}.}
\end{equation}
	Similarly, using the  identity $\cosh(x+y) = \cosh x \cosh y +\sinh x \sinh y$ yields the lower bound
	\begin{equation}\label{eq:rho_lower_bd}
	{	\cosh \rho(s;R)  \geq \frac{\cosh(R - \Delta)}{\cosh(s-\Delta)}.}
	\end{equation}
	Together, \eqref{eq:rho_upper_bd} and \eqref{eq:rho_lower_bd} prove \eqref{eq:bound_rho_arcosh}.
{	Finally, we note that when $s \geq -R + 2\Delta$ (whence $ \frac{\cosh (R-\Delta)}{\cosh (s-\Delta)}  \geq 1$), \eqref{eq:bound_rho_arcosh} implies
	\[
	\arcosh \frac{\cosh (R-\Delta)}{\cosh (s-\Delta)} \leq  \rho(s;R) \leq \arcosh \frac{\cosh (R+\Delta)}{\cosh (s-\Delta)}.
	\]
	 Now \eqref{eq:bound_rho} follows, for $s \geq -R + 2\Delta$, from this coupled with the following estimates for the inverse hyperbolic cosine (see \cite[Lemma 14]{HeroldHugThaele}): for $|a| \leq b$ one has
	\[
	b-|a| \leq \arcosh \frac{\cosh b}{\cosh a} \leq b-|a| + \log 2.
	\]
	In the complementary case $s < -R+2\Delta$, the right-hand-side of \eqref{eq:bound_rho} follows in a similar way, and the left-hand-side holds trivially since $R  -  \Delta - |s-\Delta| \leq 0$}. This completes the proof.
\end{proof}

Finally, let us prove Lemma \ref{lem:intersection_asymp} regarding the intersection volume.
	
	\begin{proof}[Proof of Lemma \ref{lem:intersection_asymp}]
		First we prove the upper bound. This is done by combining the upper bound on $\rho(s;R)$ provided by \eqref{eq:bound_rho_arcosh} with the elementary inequality (for $d > 2$)
		\begin{equation*}
		\int_0^\rho \sinh^{d-2} (u) \,\dint u \leq \frac{\cosh^{d-2}(\rho)}{d-2}
		\end{equation*}
		This  gives
		\begin{align*}
		\cH^{d-1}(H(s) \cap B_R^d) &\leq \frac{\omega_{d-1}}{\mu^{d-1}} \frac{1}{d-2}  \left[\frac{\cosh (R+\Delta)}{\cosh(s-\Delta)}\right]^{d-2},
		\end{align*}
		which is \eqref{eq:inter_vol_bound}. To prove \eqref{eq:inter_vol_asymp}, se apply the asymptotics
		$\arcosh t = \log(2t) + o(1)$, as $t \to \infty$,
		to the expression \eqref{eq:rho} for $\rho(s;R)$ to obtain the asymptotics, as $R \to \infty$ and for fixed $s$
		\begin{align*}
		\rho (s;R) &= \log \left(2 \frac{\cosh R - \sinh \Delta \sinh(s-\Delta)}{\cosh \Delta \cosh(s-\Delta)}\right) + o(1) \\
		& = \log(2 \cosh R) - \log \cosh(s-\Delta) - \log \cosh \Delta + o(1) \\
		& = R - \log \cosh(s-\Delta)+ \log \mu + o(1),
		\end{align*}
		where we have used that $\cosh \Delta = \mu^{-1}$.
		
		In view of the expression \eqref{eq:vol_intersection} for the intersection volume, using \eqref{eq:hyper_int_asymp} we conclude that
		\begin{align*}
		\cH^{d-1}(H(s)\cap B_R^d) &\sim {\omega_{d-1}\over \mu^{d-1}}\cdot{1\over(d-2)2^{d-2}}e^{(d-2)(R+\log\mu-\log\cosh(s-\Delta))+o(1)}\\
		&\sim \frac{\omega_{d-1}}{(d-2)2^{d-2}} \cdot  \mu^{-1} e^{(d-2)R} \cosh^{-(d-2)}(s-\Delta),
		\end{align*}
		as asserted.
\end{proof}

\section{Proofs}

\subsection{Cumulant estimates}\label{subsec:cumulants}

In the proofs below, the following integrals will play a key role
\begin{equation}
I_{\lambda,k}(R) = \int_{\Hyp_\lambda} \bigl(\cH^{d-1} (H \cap B_R^d))^k \,\Lambda_\lambda(\dint H),
\end{equation}
for $k \in \NN$ and $\lambda \in [0,1]$. Let us mention that $I_{k,\lambda}(R)$ is nothing but the $k$-th cumulant of the random variable $S_R^{(\lambda)}$ (see \cite[Corollary 1]{LPST}), although we will not need this fact, except for the simple cases $k=1,2$. We will repeatedly use the following estimates on $I_{k,\lambda}(R)$ for $k \geq 2$.

\begin{proposition}\label{prop:cumulants}
\begin{enumerate}[(i)]
	\item\label{item:cum_dim_2} Suppose that $\lambda \in [0,1)$ and  $d=2$. Then there are constants $A_k,B_k\in(0,\infty)$ for $k \geq 2$, only depending on $k$ and a constant $R_{k,\lambda}$ depending only on $k$ and $\lambda$ such that
	$$
	A_k (1-\lambda)^{-(k/2-1)}e^R \leq I_{\lambda,k}(R) \leq B_k {(1-\lambda)^{-k/2}} e^R
	$$
	holds for $R>R_{k, \lambda}$.
	
	\item Suppose that $\lambda \in [0,1)$ and $d=3$. Then there are constants $A_k \in (0,\infty)$ for $k \geq 2$, only depending on $k$,   such that as $R \to \infty$
	\begin{equation*}
	I_{\lambda,k}(R) \sim \begin{cases}
	A_k \, (1-\lambda^2)^{1-k/2} \, e^{kR} &: k > 2, \\
	A_2 \, R \, e^{2R} &: k = 2.
	\end{cases}
	\end{equation*}
	
	\item Suppose that $\lambda \in [0,1)$ and $d \geq 4$. Then there are constants $A_{k ,d}\in (0,\infty)$ for $k \geq 2$, only depending on $k$ and $d$, such that as $R \to \infty$
	\begin{equation*}
	I_{\lambda, k}(R) \sim A_{k,d} \,(1-\lambda^2)^{(d-1-k)/2} \,e^{k(d-2)R}
	\end{equation*}
	
	\item Finally, suppose that $\lambda = 1$. Then for any $d \geq 2$, there exist constant $A_{k,d} \in (0,\infty)$ for $k \geq 2$, depending only on $k$ and $d$, such that
	\begin{equation*}
	I_{1,k}(R) \sim \begin{cases}
	A_{k,d} \,e^{(k-1)(d-1)R} &: k > 2 \\
	 A_{2,d} \, R \,e^{(d-1)R} &: k = 2.
	\end{cases}
	\end{equation*}
\end{enumerate}
\end{proposition}

\begin{proof}
	We first suppose that $\lambda \in [0,1)$. We start with the case $d=2$.
	Observe that the formula \eqref{eq:vol_intersection} in this case reads
		\begin{equation*}
		\cH^1(H(s) \cap B_R^2) = \frac{2}{\mu} \rho(s;R),
		\end{equation*}	
		and so, using the expressions \eqref{eq:dHlambda} and~\eqref{eq:density_Delta} for the invariant measure, we have
		\begin{equation*}
		I_{\lambda, k}(R) = 2^k \mu^{1-k} \int_{-R}^{R} \rho(s;R)^k \cosh(s-\Delta) \,\dint s.
		\end{equation*}
		Let us begin with proving the upper bound. Denoting $M := R + \Delta + \log 2$, from \eqref{eq:bound_rho} we get $\rho(s;R) \leq M- |s-\Delta|$. Putting $u=s-\Delta$ and $z=M-u$, this gives
		\begin{align*}
		I_{\lambda,k}(R) & \leq 2^k \mu^{1-k} \int_{-R}^{R + 2 \Delta} (M - |s-\Delta|)^k \cosh (s-\Delta) \,\dint s \\
		& \leq 2^{k+1} \mu^{1-k} \int_{0}^{R+\Delta} (M-u)^k e^{u} \,\dint u \\
		&  \leq 2^{k+1} \mu^{1-k} e^M \int_0^\infty z^k e^{-z} \,\dint z \\
		& = 2^{k+2} k! \, \mu^{1-k} e^{\Delta} e^R.
		\end{align*}
		Finally, noting that $\mu^{1-k} e^{\Delta}  = (1 + \lambda)^{1-k/2} (1-\lambda)^{-k/2} \leq (1-\lambda)^{-k/2}$ gives the required upper bound.
		
		The proof of the lower bound is similar. Assuming $R > \Delta $, we get from \eqref{eq:bound_rho} the lower bound
		$
		\rho(s;R) \geq R  - \Delta- |s-\Delta|
		$
		(note that this lower bound is non-negative only for $|s-\Delta| \leq R - \Delta$). Therefore, proceeding as above and by using the substitutions $u=s-\Delta$ and $z=R-\Delta-u$ we get
		\begin{align*}
		I_{\lambda,k}(R) &\geq  2^k \mu^{1-k} \int_{-(R-2\Delta)}^R (R-\Delta-|s-\Delta|)^k \cosh (s-\Delta) \,\dint s  \\
		&\geq  2^{k} \mu^{1-k} \int_{0}^{R-\Delta} (R-\Delta-u)^k e^u \,\dint u \\
		& = 2^{k} \mu^{1-k} e^{-\Delta} e^R \int_0^{R-\Delta} z^k e^{-z} \,\dint z.
		\end{align*}
		Using this time $\mu^{1-k} e^{-\Delta}\geq 2^{-k/2} (1-\lambda)^{-(k/2-1)}$, it remains to choose $R_{k,\lambda}$ so that $\int_0^{R_{k,\lambda}-\Delta} z^k e^{-z} \,\dint z >\frac{1}{2}$ to obtain the lower bound. This proves $(i)$.
	
		\medskip
	
		Next consider the case $d = 3$. We note that in this case, \eqref{eq:vol_intersection} gives
		\begin{align*}
		\cH^2(H(s) \cap B_R^3) &= \frac{2\pi}{\mu^2} \int_0^{\rho(s;R)} \sinh u \,\dint u = \frac{2\pi}{\mu^2} \left(\cosh \rho(s;R)-1\right) = \frac{2\pi}{\mu}  \, \frac{\cos R - \cosh s}{ \cosh(s-\Delta)} ,
		\end{align*}
		where in the last equality we used the definition \eqref{eq:rho} of $\rho(s;R)$ together with the hyperbolic identity $\cosh(x+y) = \cosh x \cosh y + \sinh x \sinh y$ and the fact that $\cosh \Delta = \mu^{-1}$. Using the expressions \eqref{eq:dHlambda} and~\eqref{eq:density_Delta} for the invariant measure we get
		\begin{equation*}
		I_{\lambda,k} (R)=  (2 \pi)^k \mu^{2-k} \int_{-R}^R \left(\cosh R - \cosh  s\right)^k \cosh^{-(k-2)}(s-\Delta) \,\dint s.
		\end{equation*}
		When $ k > 2$, this gives
		\begin{equation*}
		I_{\lambda,k} (R) =  (2 \pi)^k \mu^{2-k} \cosh^k(R) \int_{-R}^R \left(1-\frac{\cosh s}{\cosh R}\right)^k \cosh^{-(k-2)}(s-\Delta) \,\dint s.
		\end{equation*}
		Observe that the latter integrals tend as $R \to \infty$ to a finite limit, using the monotone convergence theorem and the integrability of the function $s \mapsto \cosh^{-(k-2)}(s-\Delta)$ on $\RR$. This proves the asymptotics for $k > 2$. When $k=2$ we get
		\begin{equation}\label{eq:I_lambda_2_d_3}
		I_{\lambda,2} (R) =  (2 \pi)^2  \int_{-R}^R \left(\cosh R - \cosh  s\right)^2\dint s
		 = (2 \pi)^2 \left[  2R \cosh^2 R  - 3 \sinh R \cosh R + R  \right],
		\end{equation}
		and clearly for large $R$ the first term is dominant, proving the required asymptotics. This proves $(ii)$.
		
		\medskip
		
		Next we consider the case $d \geq 4$. Denoting $g_R^{(\lambda)}(s) := e^{-(d-2)R} \cdot  \cH^{d-1}(H(s) \cap B_R^d)$, we can write
			\begin{align}\label{eq:I_k_d_ge_4}
		I_{\lambda, k}(R) = e^{k(d-2)R} \mu^{d-1} \int_{-R}^R g_R^{(\lambda)}(s)^k \cosh^{d-1}(s-\Delta) \,\dint s.
		\end{align}
		Now, Lemma \ref{lem:intersection_asymp} implies the pointwise limit
		\begin{equation*}
		g^{(\lambda)}_R(s) \to g^{(\lambda)}(s) :=  \frac{\omega_{d-1}}{(d-2)2^{d-2}} \cdot \mu^{-1} \cosh^{-(d-2)}(s-\Delta).
		\end{equation*}
  	Furthermore, the upper bound \eqref{eq:inter_vol_bound} on the intersection volumes gives
  	\begin{equation*}
  	g^{(\lambda)}_R(s) \leq C \cosh^{-(d-2)}(s-\Delta),
  	\end{equation*}
  	where the constant $C$ depends only on $d$ and $\lambda$. This gives an integrable upper bound for the integrands in \eqref{eq:I_k_d_ge_4} (note that $d-1-k(d-2) < 0$), and so the dominated convergence theorem implies that
	\begin{equation*}
		I_{\lambda, k}(R) \sim \left( \frac{\omega_{d-1}}{(d-2)2^{d-2}} \right)^k \mu^{d-k-1} e^{k(d-2)R} \int_{-\infty}^\infty \cosh^{d-1-k(d-2)}(s) \,\dint s,
	\end{equation*}
	which proves $(iii)$.
			
		\medskip
		
		Finally, we consider the case $\lambda = 1$.  Using the expression \eqref{eq:vol_intersection_horo} for the intersection volume and noting that the invariant measure $\Lambda_1$ has density $e^{-(d-1)s}$, we have in this case
		\begin{align}
		I_{1, k} (R) = (2^{(d-1)/2}\kappa_{d-1})^k \int_{-R}^R (\cosh R - \cosh s)^{(k/2)(d-1)}e^{(k/2-1)(d-1)s}\,\dint s \label{eq:I_k_horo}
 		\end{align}
 		When $k=2$, \eqref{eq:I_k_horo} simplifies to
 		\begin{align*}
 		I_{1, 2} (R) &= 2^{d-1}\kappa_{d-1}^2 \int_{-R}^R (\cosh R - \cosh s)^{d-1}\,\dint s \\
 		    			& =2^{d-1}\kappa_{d-1}^2\cdot  \sum_{j=0}^{d-1} \binom{d-1}{j}\cosh^{d-1-j}(R) (-1)^j \int_{-R}^R \cosh^{j}(s) \,\dint s \\
  						&= 2^{d-1}\kappa_{d-1}^2 \left[2R \cosh^{d-1}(R) + \sum_{j=1}^{d-1} \binom{d-1}{j} \cosh^{d-1-j}(R) (-1)^j \int_{-R}^R \cosh^j(s) \,\dint s\right].
 		\end{align*}
 		Noting that each summand in the second sum is $O(\cosh^{d-1}(R))$ gives the desired asymptotics. When $k > 2$, we rewrite \eqref{eq:I_k_horo} as
 		\begin{equation*}
 		\begin{split}
 		I_{1, k} (R)  &=  (2^{d-1 \over 2}\kappa_{d-1})^k\,\cosh^{{k \over 2}(d-1)}(R) \,e^{ \left( {k \over 2}-1\right)(d-1)R} \\
 		&\qquad\qquad\times\int_{-R}^R \left(1 - \frac{\cosh s}{\cosh R}\right)^{{k \over 2}(d-1)}e^{ \left({k \over 2}-1\right)(d-1)(s-R)}\,\dint s.
 		\end{split}	
 		\end{equation*}
 		In the last integral, performing the substitution $z = e^{s-R}$, and using the hyperbolic identity $\cosh(x+y) = \cos x \cos y + \sinh x \sinh y$ transforms the integral into
 			\begin{align*}
 		\int_{e^{-2R}}^1 \left(1-z + (1-\tanh R) {z + 1/z \over 2}\right)^{{k \over 2} (d-1)} z^{\left({k \over 2}-1\right)(d-1)-1} \,\dint z,
 		\end{align*}
 		which by the dominated convergence theorem tends to a finite limit, specifically the Beta integral $B\big({k \over 2}(d-1) + 1, ({k \over 2}-1)(d-1)\big)$ . This completes the proof of $(iv)$, and hence of the proposition.				
	\end{proof}

\subsection{Proof of Theorem \ref{thm:exp_var}}

Here we prove Theorem \ref{thm:exp_var}. We first note that
\begin{align}
\EE S_R^{(\lambda)} &= \int_{\Hyp_\lambda} \cH^{d-1}(H \cap B_R^{d}) \, \Lambda_\lambda(\dint H) = I_{\lambda, 1} (R) \label{eq:exp_I1} \\
\intertext{and}
\Var S_R^{(\lambda)} & = \int_{\Hyp_\lambda}\bigl( \cH^{d-1}(H \cap B_R^{d}))^2\, \Lambda_\lambda(\dint H) = I_{\lambda, 2} (R).\label{eq:var_I2}
\end{align}
While these are, as mentioned before, special cases of \cite[Corollary 1]{LPST}, they are much simpler and follow from {the multivariate Mecke equation (see, e.g., \cite[Theorem 4.4]{LPBook}).}

\begin{proof}[Proof of Theorem \ref{thm:exp_var}]
Using \eqref{eq:exp_I1}, the expectation computation follows at once from the Crofton formula for $\lambda$-geodesic hyperplanes \cite[Proposition 3.1]{Sol}. In view of \eqref{eq:var_I2}, the statements about the variances are immediate consequences of Proposition \ref{prop:cumulants} (with $k=2$).
\end{proof}

\subsection{Proof of Theorem \ref{thm:NonCltlambda} (i)}\label{subsec:ProofLambdad<4}
To prove Theorem \ref{thm:NonCltlambda} (i), we apply a general normal approximation bound for Poisson $U$-statistics, which in our setting reads as follows. Recall the integrals
\begin{equation}
I_{\lambda,k}(R) = \int_{\Hyp_\lambda} \cH^{d-1} (H \cap B_R^d)^k \,\Lambda_\lambda(\dint H)
\end{equation}
introduced in Section \ref{subsec:cumulants}. Then \cite[Theorem 4.7]{RS} and \cite[Theorem 4.2]{SchulteKolmogorov} applied to $S_R^{(\lambda)}$ yield that
\begin{equation}\label{eq:Wass_bound}
d_\star\left(\frac{S_R^{(\lambda)} - \EE S_R^{(\lambda)}}{\sqrt{\Var S_R^{(\lambda)}}},N\right) \leq c_\star  \frac{\sqrt{I_{\lambda,4}(R)}}{I_{\lambda,2}(R)},
\end{equation}
where $\star\in\{\Wass,\Kol\}$, $N$ is a standard Gaussian random variable and $c_\star\in(0,\infty)$ is a constant only depending on the choice of $\star$.

\begin{proof}[Proof of Theorem \ref{thm:NonCltlambda} (i)]
	First we consider the case $d=2$. Then, combining the bound \eqref{eq:Wass_bound} with Proposition \ref{prop:cumulants} $(i)$ gives, for sufficiently large $R$,
	\begin{align*}
	d_\star\left(\frac{S_R^{(\lambda)} - \EE S_R^{(\lambda)}}{\sqrt{\Var S_R^{(\lambda)}}},N\right) &\leq c\, \frac{\sqrt{(1-\lambda)^{-2} e^R   }}{  e^R} = c\, (1-\lambda)^{-1} e^{-R/2}
	\end{align*}
	for some absolute constant $c\in(0,\infty)$. This proves Theorem \ref{thm:NonCltlambda} $(i)$ for $d=2$.
	
	Similarly, for $d=3$, we combine the bound \eqref{eq:Wass_bound} with Proposition \ref{prop:cumulants} $(ii)$ to get, for sufficiently large $R$,
	\begin{equation*}
	d_\star\left(\frac{S_R^{(\lambda)} - \EE S_R^{(\lambda)}}{\sqrt{\Var S_R^{(\lambda)}}},N\right) \leq c\, \frac{\sqrt{ (1-\lambda^2)^{-1} e^{4R}   }}{  R \, e^{2R}} \leq c\, (1-\lambda)^{-\frac{1}{2}} R^{-1},
	\end{equation*}
	for some absolute constant $c\in(0,\infty)$. This completes the proof of Theorem \ref{thm:NonCltlambda} $(i)$.
\end{proof}

\subsection{Proof of Theorem \ref{thm:NonCltlambda} (ii)}\label{sec:ProofLambdad>4}

Since the overall strategy of the proof is the same as that of Theorem \ref{thm:NonClt}, we only point out the essential differences. As in the beginning of the proof of Theorem \ref{thm:NonClt}, we can compute the characteristic function of the centred and normalized random variable $S_R^{(\lambda)}$.
Using the expression~\eqref{eq:dHlambda} for the invariant measure combined with~\eqref{eq:density_Delta}, we have that
\begin{align*}
	\psi_R^{(\lambda)}(t):=\EE e^{\mathfrak{i}t{S_R^{(\lambda)}-\EE S_R^{(\lambda)}\over e^{R(d-2)}}} = \exp\Big(\mu^{d-1}\int_{-R}^R(e^{\mathfrak{i}tg_R^{(\lambda)}(s)}-1-{\mathfrak{i}tg_R^{(\lambda)}(s)})\cosh^{d-1} (s-\Delta)\,\dint s\Big)
\end{align*}
for $t\in\RR$ with
$$
g_R^{(\lambda)}(s):=e^{-R(d-2)}\cH^{d-1}(H(s)\cap B_R^d)
$$
and where $H(s)$ stands for a $\lambda$-geodesic hyperplane at oriented distance $s$ from the origin $\bo$.

From Lemma \ref{lem:intersection_asymp}, we know that for fixed $s$, one has
\begin{equation*}
\lim_{R\to\infty} g_R^{(\lambda)} (s) =  g^{(\lambda)} (s) :=\frac{\omega_{d-1}}{(d-2)2^{d-2}} \cdot  \mu^{-1} \cosh^{-(d-2)}(s-\Delta).
\end{equation*}
From this point on the proof of Theorem \ref{thm:NonCltlambda} (ii) is the same as that of Theorem \ref{thm:NonClt} with the obvious modifications. An integrable upper bound is obtained from the upper bound \eqref{eq:inter_vol_bound} on the intersection volume, which implies that
\begin{equation*}
g^{(\lambda)}_R(s) \leq C \cosh^{-(d-2)}(s-\Delta),
\end{equation*}
where the constant $C$ depends only on $d$ and $\lambda$.
Using this, we get
\begin{align*}
\left|(e^{\mathfrak{i}tg_R^{(\lambda)}(s)}-1-\mathfrak{i}tg_R^{(\lambda)}(s))\cosh^{d-1} (s-\Delta)\right| &\leq \frac{t^2}{2} g_R^{(\lambda)}(s)^2 \cosh^{d-1}(s-\Delta)\\
& \leq C^2 \frac{t^2}{2} \cosh^{-(d-3)}(s-\Delta),
\end{align*}
which is integrable on $\RR$ for $d \geq 4$. Therefore, we have the convergence
\begin{align*}
\lim_{R\to\infty}\psi_R^{(\lambda)}(t)
= \psi(t) :=
\exp\Big(\mu^{d-1}\int_{-\infty}^\infty(e^{\mathfrak{i}tg^{(\lambda)}(s)}-1-\mathfrak{i}tg^{(\lambda)}(s))\cosh^{d-1}(s-\Delta)\,\dint s\Big)
\end{align*}
of the characteristic functions, for all $t\in\RR$. Using the substitution $w= s-\Delta$ and the notation $h(w) = \cosh^{-(d-2)}(w)$, $w>0$, we can write the above as follows:
$$
\psi\left(\frac{(d-2)2^{d-2}\mu}{\omega_{d-1}} t\right)
=
\exp\Big( 2 \mu^{d-1} \int_0^{\infty}(e^{\mathfrak{i}t h(w)}-1-\mathfrak{i}t h(w))\cosh^{d-1}(w)\,\dint w\Big).
$$
As in the proof of Theorem \ref{thm:NonClt}, we can see that the right-hand side is the characteristic function $\EE e^{\mathfrak{i} t Z_{d,\lambda}}$ of the random variable $Z_{d,\lambda}$ appearing in Theorem~\ref{thm:NonCltlambda} (ii). This completes the argument.\qed

\subsection{Proof of Theorem \ref{thm:CLT_horosphere}}\label{sec:ProofHoro}

We start by recalling the variance of the random variables $S^{(1)}_R$ and its  asymptotic behaviour, as $R\to\infty$. Indeed, since by \eqref{eq:var_I2},
\begin{equation*}
\Var S_R^{(1)} = I_{1,2}(R),
\end{equation*}
we deduce from Proposition \ref{prop:cumulants} $(iv)$ (and its proof) that
\begin{equation}\label{eq:var_hor_exact}
\Var S_R^{(1)} = 2^{d-1} \kappa_{d-1}^2 \int_{-R}^R (\cosh R - \cosh s)^{d-1} \,\dint s
\end{equation}
and moreover,
\begin{equation}\label{eq:var_hor_asymp}
\Var S_R^{(1)}  \sim 2\kappa_{d-1}^2 R e^{(d-1)R}, \quad \text{ as $R\to\infty$.}
\end{equation}

%
%

We are now in a position to present the proof of Theorem \ref{thm:CLT_horosphere}.

\begin{proof}[Proof of Theorem \ref{thm:CLT_horosphere}]
In view of Proposition \ref{prop:intersection}, we have the representation
\begin{equation*}
S_R^{(1)} = \sum_{s \in \xi} f_R^{(1)}(s),
\end{equation*}
where $\xi$ is an inhomogeneous Poisson process on $\RR$ with density $s \mapsto (\cosh s - \sinh s)^{d-1} = e^{-(d-1)s}$, and the function $f_R^{(1)}$ is defined by $f_R^{(1)}(s) = \kappa_{d-1} \bigl[ 2 e^s (\cosh R - \cosh s)\bigr]^{\frac{d-1}{2}}$ for $|s| \leq R$, and $f_R^{(1)}(s) =0$ otherwise. We decompose the random variable $S_R^{(1)}$ into `positive' and `negative' parts as follows:
\begin{equation*}
S_R^{(1)} = S_{R,+}^{(1)} + S_{R,-}^{(1)},
\end{equation*}
where
\begin{equation*}
S_{R, +}^{(1)} = \sum_{\substack{s \in \xi \\  s > 0}} f_R^{(1)}(s) \qquad \text{and} \qquad S_{R, -}^{(1)} = \sum_{\substack{s \in \xi \\  s < 0}} f_R^{(1)}(s).
\end{equation*}
Observe that $ S_{R,+}^{(1)}$ and $S_{R,-}^{(1)}$ are independent, due to the independence properties of a Poisson process. Denoting $\sigma_R^2 := \Var S_R^{(1)}$, we may then write
\begin{equation*}
\frac{S_R^{(1)} - \EE S_R^{(1)} }{\sqrt{\Var S_R^{(1)} }} = \frac{S_{R,+}^{(1)} - \EE S_{R,+}^{(1)}}{\sigma_R} + \frac{S_{R,-}^{(1)} - \EE S_{R,-}^{(1)}}{\sigma_R}.
\end{equation*}
Our plan is now to prove that the first summand converges to zero in distribution, while the second converges to the required Gaussian random variable $N_{1\over 2}$.
This suffices to prove the theorem and explains the non-standard variance $1/2$: both terms, $S_{R,+}^{(1)}$ and $S_{R,-}^{(1)}$, contribute to the variance, but only one of these terms contributes to the fluctuations because the other one is killed by the normalization $\sigma_R$.

We begin with the first summand. We first compute
\begin{align*}
\EE S_{R,+}^{(1)} =& \kappa_{d-1} \int_0^{R} \left[2e^s (\cosh R - \cosh s)\right]^{\frac{d-1}{2}}e^{-(d-1)s}\,\dint s \\
						  \leq &\kappa_{d-1} (2 \cosh R)^{\frac{d-1}{2}} \int_{0}^R e^{-\frac{d-1}{2}s} \,\dint s \\
						  	\leq &\kappa_{d-1} 2^{\frac {d+1}{2}} e^{\frac{d-1}{2}R}.
\end{align*}
Therefore, for any $\eps > 0$ we get
\begin{align*}
\PP \left(  \left|\frac{S_{R,+}^{(1)} - \EE S_{R,+}^{(1)}}{\sigma_R} \right| \geq \eps  \right) \leq \frac{\EE |S_{R,+}^{(1)} - \EE S_{R,+}^{(1)}|}{\eps \sigma_R} \leq \frac{4 \kappa_{d-1}}{\eps}\cdot \frac{e^{\frac{d-1}{2}R } }{\sigma_R} \xrightarrow[R \to \infty]{}0,
\end{align*}
by \eqref{eq:var_hor_asymp}. Therefore $\frac{S_{R,+}^{(1)} - \EE S_{R,+}^{(1)}}{\sigma_R}$ converges to zero in probability and, in particular, in distribution.

Now we turn to the second summand. Its characteristic function is given by
\begin{equation*}
\psi_R(t)= \EE \left[\exp\left({\mathfrak{i}t\,\frac{S_{R,-}^{(1)} - \EE S_{R,-}^{(1)}}{\sigma_R}}\right)\right]
= \exp \left(\int_{-R}^0\left(e^{\mathfrak{i} t g_R^{(1)}(s)} - 1 - \mathfrak{i} tg_R^{(1)}(s)\right) e^{-(d-1)s}\,\dint s\right),
\end{equation*}
where $g_R^{(1)}(s) := f_R^{(1)}(s) / \sigma_R$. Using Taylor expansion (see, e.g., \cite[Lemma 6.15]{KallenbergVol1}), we have
\begin{equation*}
e^{\mathfrak{i} t g_R^{(1)}(s)} - 1 - \mathfrak{i} tg_R^{(1)}(s) = -\frac{t^2}{2} g_R^{(1)}(s)^2 + E_R(s,t),
\end{equation*}
where the error term $E_R(s,t)$ satisfies the estimate
\begin{equation}
|E_R(s,t)| \leq \frac{|t|^3}{3!} g_R^{(1)}(s)^3.
\end{equation}
Therefore
\begin{align*}
\psi_R(t) = \exp\left(-\frac{t^2}{2}\, \int_{-R}^{0}g_R^{(1)}(s)^2 e^{-(d-1)s} \,\dint s\right) \exp\left(\int_{-R}^0 E_R(s,t) e^{-(d-1)s} \,\dint s \right).
\end{align*}
Consider the first term. By symmetry of the integrand and \eqref{eq:var_hor_exact}, we have
\begin{align*}
\int_{-R}^{0}g_R^{(1)}(s)^2 e^{-(d-1)s} \,\dint s
&= \phantom{\frac{1}{2}} \sigma_R^{-2} 2^{d-1} \kappa_{d-1}^2 \int_{-R}^0 (\cosh R -\cosh s)^{d-1} \,\dint s \\
&= \frac{1}{2} \sigma_R^{-2} 2^{d-1}\kappa_{d-1}^2 \int_{-R}^R (\cosh R -\cosh s)^{d-1} \,\dint s  {=} \frac{1}{2}
\end{align*}
{for all $R \geq 0$}, so that
\begin{equation*}
\lim_{R\to\infty} \exp\left(-\frac{t^2}{2}\, \int_{-R}^{0}g_R^{(1)}(s)^2 e^{-(d-1)s} \,\dint s\right)= e^{-t^2/4}.
\end{equation*}
As for the second term,  we first observe that, in view of \eqref{eq:var_hor_asymp}, we have for large enough $R>0$,
\begin{equation*}
g_R^{(1)}(s) \leq C \frac{e^{\frac{d-1}{2}(R+s)} }{\sigma_R} \leq  C  R^{-1/2} {e^{\frac{d-1}{2}s} },
\end{equation*}
for a constant $C>0$ depending only on $d$ (and changing from equation to equation). It follows that,  again for large enough $R$,
\begin{equation*}
|E_R(s,t)| \leq C |t|^3 R^{-3/2}  {e^{\frac{3}{2}(d-1)s} }.
\end{equation*}
Thus,
\begin{align*}
\left|\int_{-R}^0 E_R(s,t) e^{-(d-1)s}\,\dint s\right| & 
\leq C |t|^3 R^{-3/2 }\int_{-R}^0  e^{\frac{d-1}{2}s} \,\dint s  \leq C |t|^3 R^{-3/2}.
\end{align*}
In particular,
\begin{equation*}
\lim_{R\to\infty} \exp\left(\int_{-R}^0 E_R(s,t) e^{-(d-1)s} \,\dint s \right) = 1.
\end{equation*}
From these considerations we conclude that
\begin{equation*}
\lim_{R\to\infty} \psi_R(t) = e^{-t^2/4},
\end{equation*}
and the latter is the characteristic function of the desired centred Gaussian random variable $N_{1\over 2}$ with variance $1/2$. As noted above, this completes the proof.
\end{proof}

\section{Comparison with the Random Energy Model}\label{sec:REM}
The Random Energy Model (REM), introduced by Derrida~\cite{Derrida0,Derrida}, is probably the simplest model of a disordered system~\cite{BovierBook}. The partition function of the REM at inverse temperature $\beta>0$ is given by
$$
Z_n(\beta) = \sum_{j=1}^N e^{\beta \sqrt n X_j},
$$
where $X_1,X_2,\ldots$ are independent standard Gaussian random variables, $n$ is the number of spins in the system (each spin taking values $\pm 1$), and $N= 2^n$ is the number of spin configurations. The asymptotic behavior of $Z_n(\beta)$ as $n\to\infty$ is very well understood; see~\cite{BovierKurkovaLoewe}, \cite[Chapter~9]{BovierBook}, \cite{cranston_molchanov} and the references there.  In particular, a complete description of possible limit distributions of the appropriately centred and  normalized random variable $Z_n(\beta)$, as $n\to\infty$, was obtained in~\cite[Theorems~1.5, 1.6]{BovierKurkovaLoewe}. It turns out that for $0 < \beta < \frac 1 2 \sqrt{2\log 2}$, a central limit theorem holds for $Z_n(\beta)$, while for $\beta >\frac 12 \sqrt{2\log 2}$, the appropriately centred and normalized $Z_n(\beta)$ converges to a totally skewed $\alpha$-stable distribution with $\alpha = \sqrt{2\log 2}/ \beta$. The $\alpha$-stable distribution can be represented as a sum over a Poisson
process (in this form it appears in~\cite{BovierKurkovaLoewe}) and is similar to what we see in Theorems~\ref{thm:NonClt} and~\ref{thm:NonCltlambda} (ii).
Finally, in the boundary case $\beta = \frac 12 \sqrt{2\log 2}$, there is a CLT-type result, but the variance of the limiting normal distribution is $1/2$ rather than $1$, as in Theorem~\ref{thm:CLT_horosphere}.

This similarity between the REM and the model studied in the present paper requires an explanation.
In this section we shall give a non-rigorous treatment of both models explaining the analogy.
Before we start, several remarks are in order.
Suppose we have a sequence of random variables $(X_n)_{n\in \NN}$ such that $(X_n - a_n)/b_n$ converges to some non-degenerate limit distribution, as $n\to\infty$. Then, we say that $b_n$ is the order of fluctuations of $X_n$. It should be stressed that the order of fluctuations is not the same as the standard deviation of $X_n$ (even if the latter exists).
We shall work with exponentially growing or decaying quantities of the form $e^{c n + o(n)}$ (for the REM) and $e^{cR+o(R)}$ (for $\lambda$-geodesic hyperplanes), where $c$ stays constant. If the sequences of random variables $(X_n)_{n\in \NN}$ and $(Y_n)_{n\in \NN}$ have fluctuations of orders $e^{an}$ and $e^{bn}$ with $a>b$, respectively,  then the order of fluctuations of  $X_n + Y_n$ is $e^{an}$ (i.e., the largest exponent wins), by the Slutsky lemma.

\paragraph{Random Energy Model.}
Let us analyze the fluctuations of $Z_n(\beta)$ in a non-rigorous way. For a better comparison with $\lambda$-geodesic hyperplanes it will be convenient to poissonise the REM by letting $N\sim \text{Poi}(2^n)$ be a Poisson random variable with parameter $2^n$ that is independent of $X_1,X_2,\ldots$. Then, the points $\sqrt n X_1,\ldots, \sqrt n X_N$ form a Poisson process on $\RR$ with intensity $\lambda_n(x) = 2^n \sqrt{2\pi n}^{-1} e^{-x^2/(2n)}$, $x\in \RR$.  The property of this intensity which is crucial for us is that at $an$, where $|a| < \sqrt{2\log 2}$, the intensity \textit{grows} exponentially in $n$, namely
\begin{equation}\label{eq:REM_intensity}
\lambda_n(a n) = e^{(\log 2 - \frac {a^2} 2)n + o(n)}.
\end{equation}
Actually, the same formula holds for all $a\in \RR$, but for $|a| > \sqrt{2\log 2}$ the density \textit{decays} exponentially meaning that there are no points of the Poisson process corresponding to such $a$'s.
Let us take some $a$ with $|a| < \sqrt{2\log 2}$ and consider all points from the set $\sqrt n X_1,\ldots, \sqrt n X_N$ falling into the interval $[an, an+1]$. Any such point contributes approximately $e^{\beta an}$ to the sum $Z_n(\beta)$. The number of such points is Poisson distributed and the parameter is approximately $e^{n(\log 2 - a^2/2)}$. Since $(\text{Poi}(\mu) - \mu)/\sqrt \mu$ converges to the standard normal distribution as $\mu\to\infty$, the fluctuations of this random number are of order $e^{n(\log 2 - a^2/2)/2}$.  It follows that the contribution of those of the points $\sqrt n X_1,\ldots, \sqrt n X_{N}$ that fall into the interval $[an, an+1]$ to the fluctuations of $Z_n(\beta)$ is of order
\begin{equation}\label{eq:REM_contribution_fluct}
e^{a\beta n}e^{n(\log 2 - a^2/2)/2} = e^{\frac 12 n(2a\beta + \log 2 - a^2/2)},
\qquad
|a| <  \sqrt{2\log 2}.
\end{equation}
To guess the order of fluctuations of $Z_n(\beta)$, we have to take the maximum of these quantities over $a$ (because the larger exponent wins).  The crucial point is that the maximum should be taken only over $|a| \leq  \sqrt{2\log 2}$ because outside this interval the intensity decays exponentially and no points will be observed there, with high probability. (In contrast to this, to determine the order of the \textit{standard deviation} of $Z_n(\beta)$, it would be necessary to take into account the contribution of all real $a$).  Taking the maximum of~\eqref{eq:REM_contribution_fluct} over $|a|\leq \sqrt{2\log 2}$, we see that three cases are possible.

\vspace*{2mm}
\noindent
\textit{Case 1}. If $0 < \beta < \frac 1 2 \sqrt{2\log 2}$, then the maximum is attained at $a  = 2\beta < \sqrt{2\log 2}$. That is, the main contribution to the fluctuations of $Z_n(\beta)$ comes from the points of the Poisson process located near $2\beta n$ (more precisely, at distances $c\sqrt n$ from this point, $c\in \RR$). As already explained above, this contribution is asymptotically Gaussian. In this case, $Z_n(\beta)$ satisfies a central limit theorem.

\vspace*{2mm}
\noindent
\textit{Case 2}. If $\beta >\frac 12 \sqrt{2\log 2}$, then the maximum is attained at $a= \sqrt{2\log 2}$. This means that the main contribution to the fluctuations of $Z_n(\beta)$ comes from the upper order statistics of the sample $\sqrt n X_1,\ldots, \sqrt n X_N$ (i.e., the maximum, the second largest value, etc.)   It is well known~\cite[Corollary~4.19(iii)]{Resnick_book} that, after an appropriate centering, these upper order statistics converge to the Poisson point process on $\RR$ with intensity $e ^{-x}$, $x\in \RR$. The limiting distribution of $Z_n(\beta)$ can be expressed as a sum over this Poisson process, see~\cite{BovierKurkovaLoewe}.

\vspace*{2mm}
\noindent
\textit{Case 3}. If $\beta = \frac 12 \sqrt{2\log 2}$, then the maximum is still attained at $a= \sqrt{2\log 2}$, but this time the sum involving the Poisson process diverges. Refining the above analysis, it is possible to show that the main contribution to the fluctuation comes from the intermediate order statistics, more precisely, from the points $\sqrt n X_j$ located at $\sqrt {2\log 2}\, n - c \sqrt n$ with $c>0$ (recall that there are no points corresponding to $c<0$). On the other hand, the main contribution to the standard deviation comes from the analogous  \textit{two-sided} region with $c>0$ and $c<0$, explaining why the variance of the limiting distribution is $1/2$.  We have seen a similar phenomenon in Section~\ref{sec:ProofHoro}.

\paragraph{$\lambda$-geodesic hyperplanes.}
Let us now provide a similar non-rigorous analysis of the fluctuations of the random variable $S_R^{(\lambda)}$ defined in~\eqref{eq:S_R_lambda_def}. We restrict ourselves to the case $\lambda \in [0,1)$. By the formula~\eqref{eq:dHlambda} for the invariant measure and by~\eqref{eq:density_Delta}, the signed distances from the origin to the $\lambda$-geodesic hyperplanes form a Poisson process with intensity $\lambda(s) = \mu^{d-1} \cosh^{d-1}(s-\Delta)$, $s\in \RR$. Take some  $a\in (-1,1)$. Then, the intensity at $aR$  grows exponentially,  as $R\to\infty$,  namely
$$
\lambda(aR) = e^{(d-1)|a|  R + o(R)},
\qquad a\in (-1,1).
$$
This is similar to~\eqref{eq:REM_intensity}. If $H$ is a $\lambda$-geodesic hyperplane having signed distance $aR$ to the origin, then its contribution to $S_R^{(\lambda)}$ equals
$$
f_R(aR) = \cH^{d-1}(H\cap B_R^d) = \frac{\omega_{d-1}} {\mu^{d-1}} \int_0^{\rho(aR;R)} \sinh^{d-2} u\, \dint u.
$$
A simple computation using~\eqref{eq:rho} (and noting that $\mu\neq 0$ since we assume $\lambda\neq 1$) or Lemma~\ref{lem:bound_rho} (and noting that $\Delta$ is finite since $\lambda\neq 1$) shows that $\rho(aR;R) = (1-|a|)R + o(R)$ and hence,
$$
f_R(aR)
=
\begin{cases}
e^{(d-2)(1-|a|)R + o(R)}, & \text{ if } d\geq 3,\\
\frac 2 \mu (1-|a|)R + o(R), & \text{ if } d=2.
\end{cases}
$$
Let us now look at the contribution of the $\lambda$-geodesic hyperplanes whose distances to the origin are in the interval $[aR, aR+1]$ to the to the fluctuations of $S_R^{(\lambda)}$.
The number of such hyperplanes is Poisson distributed with parameter approximately $e^{(d-1)|a|R}$. Therefore, for $d\geq 3$, this contribution equals
$$
e^{(d-1)|a|R/2 + (d-2)(1-|a|)R} =e^{(3-d)|a|R/2 + (d-2)R}.
$$
Recalling that the largest exponent wins, we take the maximum over $a\in (-1,1)$. For $d\geq 4$, the maximum is attained at $a=0$ meaning that the main contribution to the fluctuations comes from hyperplanes having distances of order $O(1)$ to the origin. This is the reason why the Poisson process with intensity $\cosh^{d-1}(u)$ (describing the lower order statistics of the distances to the origin) appears in Theorem~\ref{thm:NonCltlambda} (ii).
{       For $d=3$, the contribution does not depend on $a$ and, indeed, an inspection of the proof of Proposition~\ref{prop:cumulants} (ii) (see, in particular, \eqref{eq:I_lambda_2_d_3}) shows that all points $aR$, $-1 <a < 1$, contribute to the fluctuations.   }
Finally, for $d=2$, the function $f_R(aR)$ grows subexponentially and therefore the contribution equals approximately $\lambda(aR) = e^{|a|  R }$. The maximum is attained at $a= \pm 1$. Since the number of hyperplanes corresponding to $a=\pm 1$ is Poissonian and the parameter grows exponentially in $R$, the limiting fluctuations are normal; see Theorem~\ref{thm:NonCltlambda} (i).


\subsection*{Acknowledgement}

The authors wish to thank Andreas Bernig (Frankfurt) for motivating us to study Poisson processes of horospheres and more general $\lambda$-geodesic hyperplanes, and the anonymous referees for useful remarks and corrections, which improved the style and clarity of the presentation.

DR and CT were supported by the German Research Foundation (DFG) via CRC/TRR 191 \textit{Symplectic Structures in Geometry, Algebra and Dynamics}. ZK and CT were supported by the German Research Foundation (DFG) via the Priority Program SPP 2265 \textit{Random Geometric Systems}. ZK was also supported by the German Research Foundation (DFG) under Germany’s Excellence Strategy EXC 2044 – 390685587 \textit{Mathematics M\"unster: Dynamics -- Geometry -- Structure}.

\addcontentsline{toc}{section}{References}

\end{document}